\documentclass[11pt]{article}
\usepackage{fullpage}
\usepackage{amsmath,amssymb,amsthm}
\usepackage[algo2e,ruled]{algorithm2e}
\usepackage{graphicx}
\usepackage{hyperref}

\newcommand{\sprod}{\mathop{\textstyle\prod}}
\newcommand{\minimize}{\mathop{\textrm{minimize}}}
\newcommand{\maximize}{\mathop{\textrm{maximize}}}
\newcommand{\argmin}{\mathop{\rm arg\,min}} 
\newcommand{\R}{\mathbf{R}} 
\newcommand{\E}{\mathbf{E}}
\newcommand{\Var}{\mathbf{Var}}
\newcommand{\cT}{\mathcal{T}}
\newcommand{\cB}{\mathcal{B}}
\newcommand{\cG}{\mathcal{G}}
\newcommand{\cS}{\mathcal{S}} 
\newcommand{\tcG}{\widetilde{\mathcal{G}}}
\newcommand{\prox}{\mathbf{prox}}
\newcommand{\dom}{\mathrm{dom\,}}

\newcommand{\ie}{\emph{i.e.}}
\newcommand{\eg}{\emph{e.g.}}

\newtheorem{theorem}{Theorem}[section]
\newtheorem{corollary}[theorem]{Corollary}
\newtheorem{lemma}[theorem]{Lemma}

\newtheorem{assumption}[theorem]{Assumption}
\newtheorem{remark}[theorem]{Remark}

\title{Multi-Level Composite Stochastic Optimization via Nested Variance Reduction}

\author{
Junyu Zhang\thanks{%
Department of Electrical Engineering, Princeton University, Princeton, New Jersey 08544; and National University of Singapore. (Email: \texttt{junyuz@princeton.edu}).}
\and
Lin Xiao\thanks{%
Work done at Microsoft Research, Redmond, WA 98052 (\texttt{lin.xiao@gmail.com}, \url{https://linxiaolx.github.io}).}
}

\begin{document}
\maketitle 

\begin{abstract} 
We consider multi-level composite optimization problems where each mapping in 
the composition is the expectation over a family of randomly chosen smooth 
mappings or the sum of some finite number of smooth mappings. 
We present a normalized proximal approximate gradient (NPAG) method where the 
approximate gradients are obtained via nested variance reduction.
In order to find an approximate stationary point where the expected norm of
its gradient mapping is less than $\epsilon$, the total sample complexity of 
our method is $O(\epsilon^{-3})$ in the expectation case, 
and $O(N+\sqrt{N}\epsilon^{-2})$ in the finite-sum case where $N$ is the 
total number of functions across all composition levels.
In addition, the dependence of our total sample complexity on the number of 
composition levels is polynomial, rather than exponential as in previous work.
\end{abstract} 

\paragraph{keywords:}
composite stochastic optimization, proximal gradient method, 
variance reduction

\section{Introduction}
\label{sec:intro}

In this paper, 
we consider multi-level composite stochastic optimization problems of
the form
\begin{equation}\label{eqn:multi-level-stochastic}
    \minimize_{x\in\R^d}\quad \E_{\xi_m}\bigl[f_{m,\xi_m}\bigl(\cdots\E_{\xi_2}\bigl[f_{2,\xi_2}\bigl(\E_{\xi_1}\!\left[f_{1,\xi_1}(x)\right]\bigr)\bigr]\cdots\bigr)\bigr] + \Psi(x),
\end{equation}
where $\xi_1,\ldots,\xi_m$ are independent random variables. 
We assume that each realization of
$f_{i,\xi_i}:\R^{d_{i-1}}\to\R^{d_i}$ 
is a smooth vector mapping (with $d_0=d$ and $d_m=1$)
and the function~$\Psi$ is convex but possibly nonsmooth.
An interesting special case is when each $\xi_i$ follows the uniform
distribution over a finite support $\{1,\ldots,N_i\}$, \ie,
\begin{equation}\label{eqn:multi-level-finite-sum}
    \minimize_{x\in\R^d}\quad \frac{1}{N_m}\sum_{j=1}^{N_m}f_{m,j}\biggl(\cdots
    \frac{1}{N_2}\sum_{j=1}^{N_2}f_{2,j}\biggl(\frac{1}{N_1}\sum_{j=1}^{N_1}f_{1,j}(x)\biggr)\cdots\biggr) + \Psi(x).
\end{equation} 
To simplify presentation, we write the above problems as 
\begin{equation}\label{eqn:smooth+convex}
    \minimize_{x\in\R^d} ~\bigl\{ \Phi(x) \triangleq F(x) + \Psi(x)\bigr\},
\end{equation}
where the smooth part of the objective possesses a multi-level compositional structure:
\begin{equation}\label{eqn:m-composition}
 F(x) = f_m\circ f_{m-1}\circ\cdots\circ f_1(x).
\end{equation}
For problems~\eqref{eqn:multi-level-stochastic} we use the definition
$f_i(\cdot) = \E_{\xi_i}[f_{i,\xi_i}(\cdot)]$  
and for problem~\eqref{eqn:multi-level-finite-sum}, we have
$f_i(\cdot)=(1/N_i)\sum_{j=1}^{N_i}f_{i,j}(\cdot)$.
 
The formulations~\eqref{eqn:multi-level-stochastic} 
and~\eqref{eqn:multi-level-finite-sum}
cover a broader range of applications beyond the classical stochastic 
programming and empirical risk minimization (ERM) problems,
which are special cases of~\eqref{eqn:multi-level-stochastic}
and~\eqref{eqn:multi-level-finite-sum} when $m=1$, respectively.
Important applications with $m=2$ include policy evaluation in reinforcement
learning and Markov decision processes
(\eg, \cite{sutton1998reinforcement,dann2014policy}),
risk-averse optimization 
(\eg, \cite{Rockafellar2007CoherentRisk,Ruszczynski2013risk-averse}),
and stochastic variational inequality
(\eg, \cite{IusemJofre2017,KoshalNedicShanbhag2013}
through a novel formulation in~\cite{GhadimiRuszWang2018}).
Other applications in machine learning and statistics include
optimization of conditional random fields (CRF), Cox's partial likelihood
and expectation-maximization; 
see formulations in \cite[Section 3.1-3.3]{BlanchetGoldfarb2017}.
Multi-level composite risk minimization problems with $m>2$ have been 
considered in, \eg, ~\cite{dentcheva2017statistical} 
and~\cite{YangWangFang2019}.
A more recent application of the multi-level problem is adversarial attacks and defenses in deep neural networks; see the survey paper \cite{YuanHeZhuLi2019}. 

\subsection{Related work}
\label{sec:previous-work}

Existing work on composite stochastic optimization traces back 
to~\cite{Ermoliev1976}.
Several recent progresses have been made for two-level ($m=2$) problems,
both for the general stochastic formulation 
(\eg, \cite{SCGD-M.Wang,ASC-PG-M.Wang,YangWangFang2019,GhadimiRuszWang2018})
and for the finite-sum formulation
(\eg, \cite{SVR-SCGD,NestedSVRG2018NeurIPS,VRSC-PG,ZhangXiao2019C-SAGA}).

In order to cite specific results, we first explain the measure of efficiency 
for composite stochastic optimization.
A well-adopted measure is their \emph{sample complexity}, 
\ie, the total number of times that any component mapping 
$f_{i,\xi_i}(\cdot)$ or Jacobian $f'_{i,\xi_i}(\cdot)$ is evaluated, 
in order to find some $\bar{x}$ (output of a randomized algorithm) such that 
$\E[\|\cG(\bar{x})\|^2]\leq\varepsilon$ for a given precision~$\varepsilon>0$.
Here $\cG(\bar{x})$ is the \emph{proximal gradient mapping} 
of the objective function~$\Phi\equiv F+\Psi$ at~$\bar{x}$
(exact definition given in Section~\ref{sec:npag}).
If $\Psi\equiv 0$, then $\cG(\bar{x})=F'(\bar{x})$ where $F'$ denotes the 
gradient of~$F$ and the criteria for $\varepsilon$-approximation becomes
$\E[\|F'(\bar{x})\|^2]\leq\varepsilon$.

For problem~\eqref{eqn:multi-level-stochastic} with two levels ($m=2$) 
and $\Psi\equiv 0$, stochastic composite gradient methods were
developed in~\cite{SCGD-M.Wang} with sample complexities 
$O(\varepsilon^{-4})$, $O(\varepsilon^{-3.5})$ and $O(\varepsilon^{-1.25})$ 
for the smooth nonconvex case, smooth convex case and smooth strongly convex 
case respectively.
For nontrivial convex~$\Psi$, improved sample complexities of 
$O(\varepsilon^{-2.25})$, $O(\varepsilon^{-2})$ and $O(\varepsilon^{-1})$ 
were obtained in~\cite{ASC-PG-M.Wang} for the three cases respectively.
These algorithms employ two-level stochastic updates in different time scales,
\ie, the two step-size sequences converge to zero at different rates.
Similar algorithms for solving multi-level problems were proposed
in~\cite{YangWangFang2019}, with sample complexity $O(\varepsilon^{-(m+7)/4})$.
A single time-scale method is proposed 
in~\cite{GhadimiRuszWang2018} for solving two-level problems, which achieves
an $O(\varepsilon^{-2})$ sample complexity, matching the result for 
single-level stochastic optimization \cite{GhadimiLan2013}.

For the composite finite-sum problem~\eqref{eqn:multi-level-finite-sum} with 
two levels ($m=2$), several recent works applied \emph{stochastic 
variance-reduction} techniques to obtain improved sample complexity.
Stochastic variance-reduction techniques were first developed for one-level 
finite-sum convex optimization problems 
(\eg, \cite{NIPS2012SAG,johnson2013accelerating,defazio2014saga,xiaozhang2014proxsvrg,Allen-ZhuYuan2016,SARAH2017ICML}), 
and later extended to the one-level nonconvex setting 
(\eg, \cite{Reddi2016SVRGnonconvex,NCVX-SAGA-Nonsmooth,Allen-ZhuHazan2016,Natasha2017ICML,Natasha2NIPS2018,SPIDER2018NeurIPS,SpiderBoost2018,SmoothSARAH2019,ProxSARAH2019})
and the two-level finite-sum nonconvex setting
(\eg, \cite{SVR-SCGD, NestedSVRG2018NeurIPS, VRSC-PG, ZhangXiao2019C-SAGA}).
For the two-level finite-sum case, the previous best known sample complexity is
$O(N+N^{2/3}\varepsilon^{-1})$ where $N=N_1+N_2$, obtained in
\cite{VRSC-PG,ZhangXiao2019C-SAGA}.

Interestingly, better complexity bounds have been obtained recently
for the non-composite case of $m=1$, specifically, 
$O(\varepsilon^{-1.5})$ for the expectation case and
$O(N_1+\sqrt{N_1}\varepsilon^{-1})$ for the finite-sum case
\cite{SPIDER2018NeurIPS,SpiderBoost2018,SmoothSARAH2019,ProxSARAH2019}.
These results look to be even better than the classical $O(\varepsilon^{-2})$ 
complexity of stochastic gradient method for single-level smooth nonconvex
optimization (\eg, \cite{GhadimiLan2013}), which is consistent with 
the statistical lower bound for sample average approximation
\cite{dentcheva2017statistical}.
Indeed, a stronger smoothness assumption is responsible for the improvement:
instead of the classical assumption that the  gradient of the expected function is Lipschitz continuous, the improved results require a 
\emph{Mean-Square-Lipschitz} (MSL) condition 
(exact definition given in Section~\ref{sec:prox-spider}).

Different from the one-level case,
a major difficulty in the multi-level cases ($m\geq 2$) is to deal with
biased gradient estimators for the composite function
(\eg, \cite{SCGD-M.Wang, ASC-PG-M.Wang, ZhangXiao2019C-SAGA}).
Unbiased gradient estimators for two-level problems are developed 
in~\cite{BlanchetGoldfarb2017}, but they do not seem to improve sample 
complexity.
On the other hand, it is shown in \cite{ZhangXiao2019CIVR}
that the improved sample complexity of
$O\bigl(\min\{\varepsilon^{-1.5}, \sqrt{N_1}\varepsilon^{-1}\}\bigr)$ 
can be obtained with biased gradient estimator for two-level problems
where the outer level is deterministic or without finite-sum. 
This result indicates that similar improvements are possible for 
multi-level problems as well.

\subsection{Contributions and outline}
In this paper, we propose stochastic gradient algorithms with nested 
variance-reduction for solving problems~\eqref{eqn:multi-level-stochastic}
and~\eqref{eqn:multi-level-finite-sum} for any $m\geq 1$, 
and show that their sample complexity for finding~$\bar{x}$ such that 
$\E[\|\cG(\bar{x})\|]\leq\epsilon$ is $O(\epsilon^{-3})$ for the expectation 
case, and $O(N+\sqrt{N}\epsilon^{-2})$ 
for the finite-sum case where $N=\sum_{i=1}^m N_i$.
To compare with previous work listed in Section~\ref{sec:previous-work},
we note that a complexity $O(\varepsilon^{-a})$ for 
$\E[\|\cG(\bar{x})\|^2]\leq\varepsilon$ 
(measure used in Section~\ref{sec:previous-work})
implies an $O(\epsilon^{-2a})$ 
complexity for $\E[\|\cG(\bar{x})\|]\leq\epsilon$ 
(by Jensen's inequality, but not vice versa); 
in other words, $\varepsilon=O(\epsilon^2)$.
In this sense, our results are of the same orders as the best known sample 
complexities for solving one-level stochastic optimization problems 
\cite{SPIDER2018NeurIPS,SpiderBoost2018,SmoothSARAH2019,ProxSARAH2019},
and they are strictly better than previous results for solving 
two and multi-level problems. 

The improvement in our results relies on a nested variance-reduction scheme
using mini-batches and a stronger \emph{uniform Lipschitz} assumption. 
Specifically, we assume that each realization of the random mapping 
$f_{i,\xi_i}$ and its Jacobian are Lipschitz continuous.
This condition is stronger than the MSL condition assumed in
\cite{SPIDER2018NeurIPS,SpiderBoost2018,SmoothSARAH2019,ProxSARAH2019}
for obtaining the
$O\bigl(\min\{\epsilon^{-3}, \sqrt{N_1}\epsilon^{-2}\}\bigr)$ 
sample complexity for one-level problems.
Nevertheless, it holds for many applications in machine learning 
(\eg, \cite{dann2014policy,BlanchetGoldfarb2017,ZhangXiao2019C-SAGA}),
especially when $\dom\Psi$ is bounded.
The results in~\cite{YangWangFang2019} also rely on such a uniform smooth condition.

Our contributions can be summarized as follows.
In Section~\ref{sec:npag}, we consider solving more general problems of the 
form~\eqref{eqn:smooth+convex} without specifying any composition structure,
and propose a Normalized Proximal Approximate Gradient (NPAG) method.
We show that so long as the Mean Square Errors (MSEs) of the approximate 
gradients at each iteration are uniformly bounded by $\epsilon^2$, then 
the iteration complexity of NPAG for finding an~$\bar{x}$ such that 
$\E\|\cG(\bar{x})\|]\leq\epsilon$ is $O(\epsilon^{-2})$.

In Section~\ref{sec:prox-spider}, we discuss stochastic variance-reduction 
techniques for constructing the approximate gradients in NPAG.
In particular, we present a simple and unified perspective on variance reduction
under the MSL condition.
This perspective clearly explains the reason for periodic restart,
and allows us to derive the optimal period length and mini-batch sizes 
without binding to any particular optimization algorithm. 
When using the SARAH/\textsc{Spider} estimator (\cite{SARAH2017ICML} and \cite{SPIDER2018NeurIPS} respectively)
in NPAG to solve one-level stochastic or finite-sum problems, we obtain a Prox-\textsc{Spider} method.
We also analyze the SVRG~\cite{johnson2013accelerating}
and SAGA~\cite{defazio2014saga} estimators in the same framework, and obtain the best-known complexities for them in the literature for nonconvex optimization.

In Section~\ref{sec:multi-level}, we present the multi-level
Nested-\textsc{Spider} method and show that 
under the uniform Lipschitz condition (stronger than MSL), 
its sample complexity for the expectation is $O(m^4\epsilon^{-3})$ 
and for the finite-sum cases is
$O(N+m^4\sqrt{N}\epsilon^{-1})$ where $N=\sum_{i=1}^m N_i$.
In particular, these results are polynomial in the number of levels~$m$,
which improve significantly over the previous result $O(\epsilon^{-(m+7)/2})$ 
in~\cite{YangWangFang2019}, where~$m$ appears in the exponent of~$\epsilon$.

In Section~\ref{sec:numerical}, we present numerical experiments to compare different variance reduction methods on two applications, with one-level and two-level composition structure respectively.
Finally, we give some concluding remarks in Section~\ref{sec:conclusion}.

\section{Normalized proximal approximate gradient method}
\label{sec:npag}

In this section, we present a normalized proximal approximate gradient (NPAG) 
method for solving problems of form~\eqref{eqn:smooth+convex},
which we repeat here for convenience:
\begin{equation} \label{eqn:min-smooth+convex}
    \minimize_{x\in\R^d} ~\left\{ \Phi(x) \triangleq F(x) + \Psi(x) \right\}.
\end{equation}
Here we consider a much more general setting, without imposing any compositional
structure on~$F$.
Formally, we have the following assumption.
\begin{assumption}
\label{assumption:Lip-F}
The functions $F$, $\Psi$ and $\Phi$ in~\eqref{eqn:min-smooth+convex} satisfy:
\begin{itemize} \itemsep 0pt
    \item[(a)] $F:\R^d\to\R$ is differentiable and its gradient $F'$ is $L$-Lipschitz continuous;
    \item[(b)] $\Psi:\R^d\to\R\cup\{+\infty\}$ is convex and lower semi-continuous; 
    \item[(c)] $\Phi\equiv F+\Psi$ is bounded below, \ie, there exists $\Phi_*$ such that $\Phi(x)\geq\Phi_*$ for all $x\in\dom \Psi$.
\end{itemize}
\end{assumption}

\smallskip

Under Assumption~\ref{assumption:Lip-F}, we define
the proximal operater of~$\Psi$, with step size~$\eta$, as
\begin{equation}\label{eqn:prox-def}
 \mathbf{prox}_{\Psi}^\eta (x)\triangleq \argmin_{y\in\R^d} \left\{\Psi(y) + \frac{1}{2\eta} \|y-x\|^2\right\}, \qquad x\in\dom\Psi.
\end{equation}
Since $\Psi$ is convex, this operator is well defined and unique for 
any $x\in\dom\Psi$.
Given the gradient of~$F$ at~$x$, denoted as $F'(x)$, 
the \emph{proximal gradient mapping} of~$\Phi$ at~$x$ is defined as
\begin{equation}\label{eqn:prox-grad-mapping}
 \mathcal{G}_\Phi^\eta(x)\triangleq \frac{1}{\eta}\Bigl(x-\mathbf{prox}_\Psi^\eta\bigl(x-\eta F'(x)\bigr)\Bigr).
\end{equation}
When $\Psi\equiv 0$, we have $\mathcal{G}_\Phi^\eta(x)=F'(x)$ for any $\eta>0$.
Additional properties of the proximal gradient mapping are given in
\cite{nesterov2013composite}.
Throughout this paper, we denote the proximal gradient mapping as $\cG(\cdot)$
by omitting the subscript~$\Phi$ and superscript~$\eta$ 
whenever they are clear from the context.

Our goal is to find an $\epsilon$-stationary point $\bar{x}$ such that 
$\|\mathcal{G}(\bar{x})\|\leq \epsilon$,
or $\E\bigl[\|\mathcal{G}(\bar{x})\|\bigr]\leq \epsilon$
if using a stochastic or randomized algorithm
(see \cite{DrusvyatskiyLewis2018MOR} for justifications of $\|\cG(\cdot)\|$
as a termination criterion).
If the full gradient oracle $F'(\cdot)$ is available, we can use the 
proximal gradient method
\begin{equation}\label{eqn:prox-grad-update}
 x^{t+1}  =  \mathbf{prox}_\Psi^\eta\bigl(x^t - \eta F'(x^t)\bigr) .
\end{equation}
By setting $\eta=1/L$ where~$L$ is the Lipschitz constant of $F'$,  
this algorithm finds an $\epsilon$-stationary point within 
$O(L\epsilon^{-2})$ iterations; see, \eg, \cite[Theorem~10.15]{Beck2017book}.

Here we focus on the case where the full gradient oracle $F'(\cdot)$ is not 
available;
instead, we can compute at each iteration~$t$ an approximate gradient $v^t$.
A straightforward approach 
is to replace $F'(x^t)$ with $v^t$ in~\eqref{eqn:prox-grad-update}, 
and then analyze its convergence properties based on certain measures of 
approximation quality such as $\E[\|v^t-F'(x^t)\|^2]$.
In Algorithm~\ref{alg:NPAG}, we present a more sophisticated variant, 
which first compute a tentative iterate 
$\tilde{x}^{t+1}=\prox_\Psi^\eta(x^t-\eta v^t)$,
and then move in the direction of $\tilde{x}^{t+1}-x^t$ with a step size
$\gamma_t$,
\ie, $x^{t+1} = x^t  + \gamma_t(\tilde{x}^{t+1} - x^t)$.
The choice of step size 
$\gamma_t=\min\left\{\frac{\eta\epsilon_t}{\|\tilde{x}^{t+1}-x^t\|},1\right\}$ 
ensures that the \emph{step length} $\|x^{t+1}-x^t\|$ is always bounded:
\begin{equation}\label{eqn:npag-step-length}
    \|x^{t+1}-x^t\|=\gamma_t\|\tilde{x}^{t+1}-x^t\|\leq\eta\epsilon_t,
\end{equation}
where $\epsilon_t$ is a pre-specified parameter.

\SetAlgoHangIndent{1.5em}
\setlength{\algomargin}{1em}
\begin{algorithm2e}[t]
	\DontPrintSemicolon
	\caption{Normalized Proximal Approximate Gradient (NPAG) Method}
	\label{alg:NPAG}
	\textbf{input:} 
	initial point $x^0\in\R^d$, proximal step size $\eta>0$, a sequence $\epsilon_t>0, \forall t$.\\
	\For{$t = 0,...,T-1$}{
        Find an approximate gradient $v^t$ at the point $x^t$.  \\[0.5ex]
		Compute $\tilde{x}^{t+1} = \mathbf{prox}_\Psi^\eta (x^t-\eta v^t).$\\
		Update $x^{t+1} = x^t  + \gamma_t(\tilde{x}^{t+1} - x^t)$ with
    $\gamma_t=\min\left\{\frac{\eta\,\epsilon_t}{\|\tilde{x}^{t+1}-x^t\|},~1\right\}$.}
    \textbf{output:} $\!\bar{x}\!\in\!\{x^0\!,...,x^{T\!-\!1}\}$, with $x^t$ chosen w.p. $p_t:=\frac{\epsilon_t}{\sum_{k=0}^{T-1}\epsilon_k}$ for $0\leq t\leq T-1$.
\end{algorithm2e} 

Algorithm~\ref{alg:NPAG} can also be written in terms of the 
\emph{approximate gradient mapping}, which we define as
\begin{equation}\label{eqn:approx-grad-mapping}
    \tcG(x^t) = \frac{1}{\eta}\bigl(x^t - \tilde{x}^{t+1}\bigr)
    =\frac{1}{\eta}\bigl(x^t-\prox_\Psi^\eta(x^t-\eta v^t)\bigr).
\end{equation}
Using this notation, the updates in Algorithm~\ref{alg:NPAG} become
$\gamma_t=\min\Bigl\{\frac{\epsilon_t}{\left\|\tcG(x^t)\right\|},1\Bigr\}$ and 
\[
    x^{t+1} 
    = x^t - \gamma_t\eta\, \tcG(x^t)
    = \left\{\begin{array}{ll}
      x^t - \eta\epsilon_t\frac{\tcG(x^t)}{\left\|\tcG(x^t)\right\|} 
      & \mbox{if}~\bigl\|\tcG(x^t)\bigr\|> \epsilon_t, \\[1.5ex]
      x^t-\eta\,\tcG(x^t) & \mbox{if}~\bigl\|\tcG(x^t)\bigr\|\leq \epsilon_t.
      \end{array} \right.
\]
Notice that when $\bigl\|\tcG(x^t)\bigr\|>\epsilon_t$, the algorithm uses the 
\emph{normalized} proximal gradient mapping with 
a step length $\eta\epsilon_t$. 
Therefore, we call it the Normalized Proximal Approximate Gradient (NPAG) 
method.

Next, we prove a general convergence result concerning
Algorithm~\ref{alg:NPAG} without specifying how the approximate gradient 
$v^t$ is generated. 
The only condition we impose is that the \emph{Mean-Square Error} (MSE), 
$\E[\|v^t-F'(x^t)\|^2]$, is sufficiently small for all~$t$.
We first state a useful lemma.

\begin{lemma}\label{lemma:descent-NPAG} 
    Suppose Assumption \ref{assumption:Lip-F} hold. 
    Then the sequence $\{x^t\}$ generated by Algorithm~\ref{alg:NPAG} satisfies
	\begin{equation*} 
	\Phi(x^{t+1}) \leq \Phi(x^t) - \left(\gamma_t/\eta - L\gamma_t^2\right)\bigl\|\tilde{x}^{t+1}-x^t\bigr\|^2 + \frac{1}{2L}\bigl\|F'(x^t)-v^t\bigr\|^2, 
    \qquad \forall\, t\geq 0.
	\end{equation*}
\end{lemma}
\begin{proof}
	According to the update rule for $x^{t+1}$, we have 
	\begin{eqnarray}
		\Phi(x^{t+1}) 
        & = & F(x^t + \gamma_t(\tilde{x}^{t+1}-x^t)) + \Psi(x^t + \gamma_t(\tilde{x}^{t+1}-x^t)) \nonumber \\
		& \leq & F(x^t) \!+\! \gamma_t\left\langle F'(x^t), \tilde{x}^{t+1}\!\!-\!x^t\right\rangle \!+\! \frac{L\gamma^2_t}{2}\|\tilde{x}^{t+1}\!\!-\!x^t\|^2 \!+\! (1\!-\!\gamma_t)\Psi(x^t) \!+\! \gamma_t \Psi(\tilde{x}^{t+1})\nonumber\\
        &=& F(x^t) + \Psi(x^t) + \frac{L\gamma^2_t}{2}\|\tilde{x}^{t+1}-x^t\|^2 
         + \gamma_t\left\langle F'(x^t)-v^t, \tilde{x}^{t+1}-x^t\right\rangle  \nonumber \\
        && +\gamma_t\Bigl(\left\langle v^t, \tilde{x}^{t+1}-x^t\right\rangle 
        +\Psi(\tilde{x}^{t+1})-\Psi(x^t)\Bigr),
        \label{eqn:Phi-upper-bound}
	\end{eqnarray}
where the inequality is due to the $L$-smoothness of~$F$ 
(see, e.g., \cite[Lemma~1.2.3]{Nesterov2018book}) 
and the convexity of~$\Psi$. 
Since $\tilde{x}^{t+1}$ is the minimizer of a $\frac{1}{\eta}$-strongly 
convex function, \ie,
\[
\tilde{x}^{t+1} = \prox_\Psi^\eta\bigl(x^t-\eta v^t\bigr)
= \argmin_{y\in\R^d}\Bigl\{\langle v^t, y-x^t\rangle + \Psi(y) + \frac{1}{2\eta}\|y-x^t\|^2\Bigr\},
\]
we have
\[
\langle v^t, \tilde{x}^{t+1}-x^t\rangle + \Psi(\tilde{x}^{t+1}) + \frac{1}{2\eta}\|\tilde{x}^{t+1}-x^t\|^2 \leq \Psi(x^t)-\frac{1}{2\eta}\|\tilde{x}^{t+1}-x^t\|^2,
\]
which implies
\[
    \left\langle v^t, \tilde{x}^{t+1}-x^t\right\rangle 
    +\Psi(\tilde{x}^{t+1})-\Psi(x^t)
    ~\leq~ -\frac{1}{\eta}\bigl\|\tilde{x}^{t+1}-x^t\bigr\|^2.
\]
Applying the above inequality to the last term in~\eqref{eqn:Phi-upper-bound}
yields
	\begin{eqnarray*}
		\Phi(x^{t+1}) 
		& \leq & \Phi(x^t) - \Bigl(\frac{\gamma_t}{\eta} -\frac{L\gamma_t^2}{2}\Bigr)\|\tilde{x}^{t+1}-x^t\|^2 + \gamma_t\langle F'(x^t)-v^t, \tilde{x}^{t+1}-x^t\rangle\\
		& \leq & \Phi(x^t) - \Bigl(\frac{\gamma_t}{\eta} -\frac{L\gamma_t^2}{2}\Bigr)\|\tilde{x}^{t+1}-x^t\|^2 + \frac{L\gamma_t^2}{2}\|\tilde{x}^{t+1}-x^t\|^2 + \frac{1}{2L}\|F'(x^t)-v^t\|^2 \\
		& =& \Phi(x^{t}) - \Bigl(\frac{\gamma_t}{\eta} - L\gamma_t^2\Bigr)\|\tilde{x}^{t+1}-x^t\|^2 + \frac{1}{2L}\|F'(x^t)-v^t\|^2,
	\end{eqnarray*}
where the second inequality is due to the Cauchy-Schwarz inequality. 
\end{proof}

We note this result is similar to that of \cite[Lemma~3]{ProxSARAH2019},
but we used a slightly looser bound for the inner product in the last step. 
This bound and our choice of $\gamma_t$ allow a simple convergence analysis 
of the NPAG method, which we present next.

\begin{theorem}\label{theorem:cvg-SNPAG} 
    Suppose Assumption~\ref{assumption:Lip-F} hold and we set $\eta=1/2L$
    in Algorithm~\ref{alg:NPAG}. 
    If it holds that
\begin{equation}\label{eqn:vt-cond}
\E\left[\|v^t-F'(x^t)\|^2\right]\leq \epsilon_t^2, \qquad \forall\, t\geq 0,
\end{equation}
for some sequence $\{\epsilon_t\}$, then the output $\bar{x}$ of Algorithm~\ref{alg:NPAG} satisfies
	\begin{equation}
	\E\bigl[\|\mathcal{G}(\bar{x})\|\bigr] \leq  \frac{4L(\Phi(x^0)-\Phi_*)}{\sum_{t=0}^{T-1}\epsilon_t} + \frac{4\sum_{t=0}^{T-1}\epsilon^2_t}{\sum_{t=0}^{T-1}\epsilon_t}.
	\end{equation}
\end{theorem}
\begin{proof}
Using Lemma~\ref{lemma:descent-NPAG}, the assumption~\eqref{eqn:vt-cond}, 
and the choice $\eta = 1/2L$, 
we obtain
	\begin{eqnarray}
		\E\left[\Phi(x^{t+1})\right] & \leq & \E\left[\Phi(x^t)\right] - \E\left[\left(\gamma_t/\eta - L\gamma_t^2\right)\bigl\|\tilde{x}^{t+1}-x^t\bigr\|^2\right] + \frac{\epsilon_t^2}{2L}\nonumber\\
        & = & \E\left[\Phi(x^t)\right] - \frac{1}{4L}\E\left[\Bigl(2\gamma_t- \gamma_t^2\Bigr)\|\tcG(x^t)\|^2\right] + \frac{\epsilon_t^2}{2L}\nonumber\\
		& \leq & \E\left[\Phi(x^t)\right] - \frac{1}{4L}\E\left[\gamma_t\|\tcG(x^t)\|^2\right]+ \frac{\epsilon_t^2}{2L},
        \label{eqn:from-lemma-expect}
	\end{eqnarray}
where the equality is due to~\eqref{eqn:approx-grad-mapping}
and the last inequality is due to $0\leq\gamma_t\leq 1$, 
which implies $2\gamma_t-\gamma_t^2\geq\gamma_t$.
Note that 
$\gamma_t=
\min\Bigl\{\frac{\epsilon_t}{\left\|\tcG(x^t)\right\|},1\Bigr\}$, 
therefore
$$\gamma_t\bigl\|\tcG(x^t)\bigr\|^2 = \epsilon^2_t\cdot\min\left\{\bigl\|\tcG(x^t)\bigr\|\epsilon_t^{-1},\, \bigl\|\tcG(x^t)\bigr\|^2\epsilon_t^{-2}\right\}\geq \epsilon_t\bigl\|\tcG(x^t)\bigr\|-\frac{1}{4}\epsilon_t^2,$$
where the inequality is due to the fact
$\min\bigl\{|z|, z^2\bigr\} \geq |z|-\frac{1}{4}$ for all $z\in\R$
(here we take $z=\epsilon_t^{-1}\|\tilde{\mathcal{G}}(x^t)\|$).
Taking expectations and substituting the above inequality 
into~\eqref{eqn:from-lemma-expect} gives
\begin{equation}
\label{eqn:new-1}
\E\left[\Phi(x^{t+1})\right] ~\leq~ \E\left[\Phi(x^t)\right] - \frac{\epsilon_t}{4L}\E\left[\bigl\|\tcG(x^t)\bigr\|\right] + \frac{9\epsilon_t^2}{16L}.
\end{equation}
On the other hand, by~\eqref{eqn:vt-cond} and Jessen's inequality we obtain
$$\E\left[\|v^t-F'(x^t)\|\right]\leq \sqrt{\E\left[\|v^t-F'(x^t)\|^2\right]}\leq\epsilon_t.$$
Denote $\hat{x}^{t+1} = \mathbf{prox}_r^\eta(x^t-\eta F'(x^t))$ and also note that $\tilde{x}^{t+1} = \mathbf{prox}_r^\eta(x^t-\eta v^t)$. 
Since the proximal operator is nonexpansive 
(see, \eg, \cite[Section~31]{Rockafellar70book}), 
we have 
\[
\E\left[\|\hat{x}^{t+1} - \tilde{x}^{t+1}\|\right] \leq \eta \E[\|v^t- F'(x^t)\|]\leq \eta\epsilon_t,  \qquad \forall\, t\geq 0.
\]
Therefore,
\begin{eqnarray*}
	\E\bigl[\|\mathcal{G}(x^t)\|\bigr] 
	=  \frac{1}{\eta}\E\bigl[\|\hat{x}^{t+1}-x^t\|\bigr] 
	&\leq& \frac{1}{\eta}\E\bigl[\|\tilde{x}^{t+1}-x^t\|\bigr] + \frac{1}{\eta}\E\bigl[\|\hat{x}^{t+1}-\tilde{x}^{t+1}\|\bigr] \\
	&\leq& \E\left[\bigl\|\tcG(x^t)\bigr\|\right]+\epsilon_t.
\end{eqnarray*} 
Combining the above inequality with \eqref{eqn:new-1}, we get 
\begin{equation}
\frac{\epsilon_t}{4L}\E\left[\bigl\|\cG(x^t)\bigr\|\right]~\leq~ \E\left[\Phi(x^t)\right] - \E\left[\Phi(x^{t+1})\right] + \frac{\epsilon^2_t}{L}.
\end{equation}
Adding the above inequality for $k=0,...,T-1$ and rearranging the terms, we obtain
\begin{equation*} 
\E\left[\bigl\|\cG(\bar x)\bigr\|\right] = \sum_{k=0}^{T-1} \frac{\epsilon_k}{\sum_{t=0}^{T-1}\epsilon_t}\cdot\E\left[\bigl\|\cG(x^t)\bigr\|\right]\leq \frac{4L(\Phi(x^0)-\Phi_*)}{\sum_{t=0}^{T-1}\epsilon_t} + \frac{4\sum_{t=0}^{T-1}\epsilon^2_t}{\sum_{t=0}^{T-1}\epsilon_t}.
\end{equation*}
This completes the proof.  
\end{proof}

\begin{remark}
	\label{remark:constant-epsilon}
	In the simplest case, one can set $\epsilon_t\equiv\epsilon$ for all $t = 0,1,...,T-1$ and $$T \geq 4L\left(\Phi(x^0)-\Phi_*\right)\epsilon^{-2}.$$ 
	Then Theorem \ref{theorem:cvg-SNPAG} directly implies
	\begin{equation*} 
	\E\left[\bigl\|\cG(\bar x)\bigr\|\right] \leq \frac{4L(\Phi(x^0)-\Phi_*)}{T\epsilon} + \frac{4T\epsilon^2}{T\epsilon} \leq 5\epsilon.
	\end{equation*}
\end{remark}

Theorem~\ref{theorem:cvg-SNPAG} applies directly to the deterministic
proximal gradient method with $v^t=F'(x^t)$. 
In this case, with $\epsilon_t$ and $T$ set according to Remark \ref{remark:constant-epsilon}, condition~\eqref{eqn:vt-cond} is automatically satisfied
and we recover the $O(L\epsilon^{-2})$ complexity of proximal gradient method
for nonconvex optimization \cite[Theorem~10.15]{Beck2017book}.

The problem of using a small, constant $\epsilon_t$ is that we will have very small step size $\gamma_t$ in the beginning iterations of the algorithm.
Although the order of the sample complexity is optimal, it may lead to slow convergence in the practice.
We will see in Section~\ref{sec:multi-level} that we can obtain the same orders of samples complexity using a variable sequence $\{\epsilon_t\}$, starting with large values and gradually decreases to the target precision $\epsilon$.
These choices of variable sequence $\{\epsilon_t\}$ lead to improved performance in practice and are adopted in our numerical experiments in Section~\ref{sec:numerical}.

In general, for the stochastic problems, we do not require the estimates $v^t$ to be unbiased.
On the other hand, condition~\eqref{eqn:vt-cond} looks to be quite restrictive
by requiring the MSE of approximate gradient to be less than 
$\epsilon_t^2$, where~$\epsilon_t$ is some predetermined precision. In the rest of this paper, we show how to ensure condition~\eqref{eqn:vt-cond}
using stochastic variance-reduction techniques and derive the total
sample complexity required for solving problems~\eqref{eqn:multi-level-stochastic} and~\eqref{eqn:multi-level-finite-sum}.

\section{A general framework of stochastic variance reduction}
\label{sec:prox-spider}

In this section, we discuss stochastic variance-reduction techniques 
for smooth nonconvex optimization. 
In order to prepare for the multi-level compositional case,
we proceed with a general framework of constructing
stochastic estimators that satisfies \eqref{eqn:vt-cond} for Lipschitz-continuous vector or matrix mappings. 
To simplify presentation, we restrict our discussion in this section to the setting where $\epsilon_t\equiv\epsilon$ in \eqref{eqn:vt-cond}, where $\epsilon$ is the target precsion so that the output $\bar x$ satisfies $\E[\|\cG(\bar x)\|]\leq O(\epsilon)$.
A more general scheme with variable sequence $\{\epsilon_t\}$ is presented in Section~\ref{sec:multi-level}.

Consider a family of mappings $\{\phi_\xi:\R^d\to\R^{p\times q}\}$ 
where the index~$\xi$ is a random variable.
We assume that they satisfy a \emph{uniform Lipschitz} condition, 
\ie, there exists $L>0$ such that 
\begin{equation}\label{eqn:uniform-Lipschitz}
    \| \phi_\xi(x) - \phi_\xi(y) \| \leq L \|x-y\|,
    \qquad \forall\, x, y\in \R^d, \qquad \forall\,\xi\in\Omega,
\end{equation}
where $\|\cdot\|$ denotes the matrix Frobenius norm
and~$\Omega$ is the sample space of~$\xi$.
This uniform Lipschitz condition implies the
Mean-Squared Lipschitz (MSL) property, \ie,
\begin{equation}\label{eqn:MSL-mapping}
    \E_\xi\bigl[\| \phi_\xi(x) - \phi_\xi(y) \|^2\bigr] \leq L^2 \|x-y\|^2,
    \qquad \forall x, y\in \R^d,
\end{equation}
which is responsible for obtaining the 
$O\bigl(\min\{\varepsilon^{-3}, \sqrt{N_1}\varepsilon^{-2}\}\bigr)$ 
sample complexity for one-level problems in
\cite{SPIDER2018NeurIPS,SpiderBoost2018,SmoothSARAH2019,ProxSARAH2019}.
This property in turn implies that the average mapping, defined as
$$\bar{\phi}(x) \triangleq \E_\xi[\phi_\xi(x)],$$ is $L$-Lipschitz. 
In addition, we make the standard assumption that they have bounded variance everywhere, \ie,
\begin{equation}\label{eqn:phi-xi-var}
\E_\xi[\|\phi_\xi(x)-\bar{\phi}(x)\|^2]\leq\sigma^2, \qquad \forall x\in\R^d.
\end{equation}
In the context of solving structured optimization problems of
form~\eqref{eqn:min-smooth+convex}, we can restrict the above assumptions to 
hold within $\dom\Psi$ instead of $\R^d$.

Suppose a stochastic algorithm generates a sequence 
$\{x^0,x^1,\ldots,\}$ using some recursive rule
\[
    x^{t+1} = \psi_t(x^t, v^t),
\]
where $v^t$ is a stochastic estimate of $\bar{\phi}(x^t)$. 
As an example, consider using Algorithm~\ref{alg:NPAG} to solve
problem~\eqref{eqn:min-smooth+convex} where the smooth part
$F(x)=\E_\xi[f_\xi(x)]$.
In this case, we have $\phi_\xi(x)=f'_\xi(x)$ and $\bar{\phi}(x)=F'(x)$,
and can use the simple estimator $v^t=f'_{\xi_t}(x^t)$ where $\xi_t$ 
is a realization of the random variable at time~$t$.
If $\epsilon\geq\sigma$, then assumption~\eqref{eqn:phi-xi-var} directly
implies~\eqref{eqn:vt-cond} with $\epsilon_t\equiv\epsilon$; thus Theorem~\ref{theorem:cvg-SNPAG} and Remark~\ref{remark:constant-epsilon}
show that the sample complexity of Algorithm~\ref{alg:NPAG}
for obtaining $\E[\|\cG(\bar{x})\|]\leq\epsilon$ is $T=O(L\epsilon^{-2})$. 

However, if the desired accuracy~$\epsilon$ is smaller than the uncertainty 
level~$\sigma$, then condition~\eqref{eqn:vt-cond} is not satisfied
and we cannot use Theorem~\ref{theorem:cvg-SNPAG} directly.
A common remedy is to use mini-batches; i.e., 
at each iteration of the algorithm, we randomly pick
a subset $\cB_t$ of~$\xi$ and set
\begin{equation}\label{eqn:mini-batch}
    v^t = \phi_{\cB_t}(x^t) \triangleq 
    \frac{1}{|\cB_t|}\sum_{\xi\in\cB_t}\phi_\xi(x^t).
\end{equation}
In this case, we have $\E[v^t]=\bar\phi(x^t)$ and 
$\E[\|v^t-\bar{\phi}(x^t)\|^2] = \sigma^2/|\cB_t|$.
In order to satisfy~\eqref{eqn:vt-cond} with $\epsilon_t\equiv \epsilon$, we need
$|\cB_t|\geq\sigma^2/\epsilon^2$ for all $t=0,\ldots,T-1$.
Then we can apply Theorem~\ref{theorem:cvg-SNPAG} and Remark~\ref{remark:constant-epsilon} to derive its
sample complexity for obtaining $\E[\|\cG(\bar{x})\|]\leq\epsilon$,
which amounts to $T\cdot\sigma^2/\epsilon^2=O(L\sigma^2\epsilon^{-4})$.

Throughout this paper, we assume that the target precision $\epsilon<\sigma$. In this case (and assuming $\Psi\equiv 0$), the total sample complexity $O(L\epsilon^{-2}+\sigma^2\epsilon^{-4})$ can be 
obtained without using mini-batches (see, e.g., \cite{GhadimiLan2013}).
Therefore, using mini-batches alone seems to be insufficient to improve sample 
complexity (although it can facilitate parallel computing; see, e.g., 
\cite{DekelGSX2012}).
Notice that the simple mini-batch estimator~\eqref{eqn:mini-batch} does not
rely on any Lipschitz condition, so it can also be used with non-Lipschitz 
mappings such as subgradients of nonsmooth convex functions.
Under the MSL condition~\eqref{eqn:MSL-mapping}, 
several variance-reduction techniques that can significantly reduce the sample 
complexity of stochastic optimization have been developed in the literature
(\eg, \cite{NIPS2012SAG,johnson2013accelerating,defazio2014saga,
xiaozhang2014proxsvrg,Allen-ZhuYuan2016,SARAH2017ICML,SPIDER2018NeurIPS}.

In order to better illustrate how Lipschitz continuity can help 
with efficient variance reduction, 
we consider the amount of samples required in constructing a number of, 
say~$\tau$, \emph{consecutive} estimates $\{v^0,\ldots,v^{\tau-1}\}$ 
that satisfy
\begin{equation}\label{eqn:unif-var-bound}
    \E\bigl[\bigl\|v^t-\bar{\phi}(x^t)\bigr\|^2\bigr] \leq \epsilon^2,
    \quad t=0,1,\ldots,\tau-1.
\end{equation}
If the incremental step lengths $\|x^t-x^{t-1}\|$ for $t=1,\ldots,\tau-1$
are sufficiently small and the period~$\tau$ is not too long,
then one can leverage the MSL condition~\eqref{eqn:MSL-mapping} 
to reduce the total sample complexity, to be much less than using the  
simple mini-batch estimator~\eqref{eqn:mini-batch} for all~$\tau$ steps.
In the rest of this section, we first study the SARAH/\textsc{Spider} estimator proposed in~\cite{SARAH2017ICML} and~\cite{SPIDER2018NeurIPS} respectively.
Then we also analyze the SVRG~\cite{johnson2013accelerating} and SAGA~\cite{defazio2014saga} estimators from the same general framework. 

\subsection{SARAH/\textsc{SPIDER} estimator for stochastic optimization}
\label{sec:spider-expect}

In order to construct~$\tau$ consecutive estimates $\{v^0,\ldots,v^{\tau-1}\}$,
this estimator uses the mini-batch estimator~\eqref{eqn:mini-batch} for $v^0$, 
and then constructs $v^1$ through $v^{\tau-1}$ using a recursion:
\begin{equation}\label{eqn:spider}
    \begin{split}
        v^0 &= \phi_{\cB_0}(x^0), \\
        v^t &= v^{t-1}  + \phi_{\cB_t}(x^t) - \phi_{\cB_t}(x^{t-1}) .
    \end{split}
\end{equation}
It was first proposed in \cite{SARAH2017ICML} as a gradient estimator of SARAH 
(StochAstic Recursive grAdient algoritHm)
for finite-sum convex optimization, and later extended to nonconvex
and stochastic optimization \cite{SmoothSARAH2019,ProxSARAH2019}.
The form of~\eqref{eqn:spider} for general Lipschitz continuous mappings
is proposed in~\cite{SPIDER2018NeurIPS}, known as
\textsc{Spider} (Stochastic Path-Integrated Differential EstimatoR).
The following lemma is key to our analysis. 
\begin{lemma}\label{lem:spider}
\emph{\cite[Lemma~1]{SPIDER2018NeurIPS}}
Suppose the random mappings $\phi_\xi$ satisfy~\eqref{eqn:MSL-mapping}.
Then the MSE of the estimator in~\eqref{eqn:spider} can be bounded as
\begin{equation}\label{eqn:spider-var-bound}
    \E\left[\bigl\|v^t-\bar\phi(x^t)\bigr\|^2\right] 
    \leq \E\left[\bigl\|v^0-\bar\phi(x^0)\bigr\|^2\right] + \sum_{r=1}^t\frac{L^2}{|\cB_r|}\E\left[\|x^r-x^{r-1}\|^2\right].
\end{equation}
\end{lemma}

If we can control the step lengths 
$\|x^t-x^{t-1}\|\leq \delta$ for $t=1,\ldots,\tau-1$, then
\begin{equation}\label{eqn:spider-seq-bound}
    \E\left[\bigl\|v^0-\bar\phi(x^0)\bigr\|^2\right] 
    \leq\cdots\leq
    \E\left[\bigl\|v^{\tau-1}-\bar\phi(x^{\tau-1})\bigr\|^2\right] 
    \leq 
    \frac{\sigma^2}{|\cB_0|} + \sum_{t=1}^{\tau-1} \frac{L^2\delta^2}{|\cB_t|}.
\end{equation}
To simplify notation, here we allow the batch sizes $|\cB_t|$ to take 
fraction values, which does not change the order of required sample complexity.
Now, in order to satisfy~\eqref{eqn:unif-var-bound}, it suffices to set
\begin{equation}\label{eqn:spider-batch-sizes}
    |\cB_0|=\frac{2\sigma^2}{\epsilon^2} \quad \mbox{and}\quad
    |\cB_t|=b = \frac{2\tau L^2\delta^2}{\epsilon^2}, \quad
    t=1,\ldots,\tau-1.
\end{equation}
In this case, the number of samples required over~$\tau$ steps is
\[
    |\cB_0| + (\tau-1)b ~\leq~  
    \frac{2\sigma^2}{\epsilon^2} + \frac{2\tau^2 L^2\delta^2}{\epsilon^2}.
\]
Therefore, it is possible to choose $\delta$ and $\tau$ to make  
the above number smaller than $\tau\sigma^2/\epsilon^2$, 
which is the number of samples required for~\eqref{eqn:unif-var-bound}
with the simple mini-batching scheme~\eqref{eqn:mini-batch}.

To make further simplifications, we choose $\delta=\epsilon/2L$, 
which leads to $b=\tau/2$ and 
$$|\cB_0| + (\tau-1)b ~\leq~ 2\sigma^2/\epsilon^2+\tau^2/2.$$
The ratio between the numbers of samples required by \textsc{Spider} and 
simple mini-batching is
\[
    \rho(\tau) ~\triangleq~ \left(\frac{2\sigma^2}{\epsilon^2} + \frac{\tau^2}{2} \right) \bigg/ \left(\tau\frac{\sigma^2}{\epsilon^2}\right)
    ~=~ \frac{2}{\tau} + \frac{\tau \epsilon^2}{2\sigma^2} .
\]
In order to maximize the efficiency of variance reduction, 
we should choose $\tau$ to minimize this ratio, which is done by setting
$$\tau_\star= 2\sigma/\epsilon.$$
Using~\eqref{eqn:spider-batch-sizes} and $\delta=\epsilon/2L$, 
the corresponding mini-batch size and optimal ratio of saving are
\begin{equation}\label{eqn:opt-batch-ratio}
b_\star=\tau_\star/2 = \sigma/\epsilon, \qquad
\rho(\tau_\star) = 2\,\epsilon/\sigma.
\end{equation}
Therefore, significant reduction in sample complexity can be expected
when $\epsilon\ll\sigma$. 

\begin{algorithm2e}[t]
	\DontPrintSemicolon
    \caption{Prox-\textsc{Spider} Method}
	\label{alg:prox-spider}
	\textbf{input:} 
	initial point $x^0\in\R^d$, proximal step size $\eta>0$, and precision $\epsilon>0$; \\
    \hspace{8ex}\textsc{Spider} epoch length~$\tau$, large batch size $B$ and small batch size~$b$. \\
	\For{$t = 0, 1,...,T-1$}{
            \eIf{$t \bmod \tau ==0$}{
                Randomly sample a batch $\cB^t$ of~$\xi$ with $|\cB^t|=B$\\
                $v^t=f'_{\cB^t}(x^t)$
            }{
            Randomly sample a batch $\cB^t$ of~$\xi$ with $|\cB^t|=b$\\
                $v^t=v^{t-1}+f'_{\cB^t}(x^t)-f'_{\cB^t}(x^{t-1})$
            }
            $\tilde{x}^{t+1} = \mathbf{prox}_\Psi^\eta (x^t-\eta v^t)$\\
            $x^{t+1} = x^t  + \gamma_t(\tilde{x}^{t+1} - x^t)$ with
$\gamma_t=\min\left\{\frac{\eta\,\epsilon}{\|\tilde{x}^{t+1}-x^t\|},~1\right\}$
    }
    \textbf{output:} 
    $\bar{x}$ uniformly sampled from $\{x^0,...,x^{T-1}\}$.
\end{algorithm2e} 

The above derivation gives an optimal length~$\tau_\star$ for variance
reduction over consecutive estimates $\{v^0,\ldots,v^{\tau-1}\}$.
If the number of iterations required by the algorithm is lager 
than~$\tau_\star$, then we need to restart the \textsc{Spider} estimator 
every~$\tau_\star$ iterations.
This way, we can keep the optimal ratio~$\rho(\tau_\star)$ 
to obtain the best total sample complexity.

Algorithm~\ref{alg:prox-spider} is an instantiation of the NPAG method
(Algorithm~\ref{alg:NPAG}) that uses \textsc{Spider} to construct $v^t$. 
It incorporates the periodic restart scheme dictated by the analysis above.
The \textsc{Spider}-SFO method (Option II) in~\cite{SPIDER2018NeurIPS} 
can be viewed as a special case of Algorithm~\ref{alg:prox-spider} 
with $\Psi\equiv 0$:
\begin{equation}\label{eqn:spider-sfo}
    x^{t+1} = x^t - \gamma_t \eta\, v^t, \quad \mbox{where}~
    \eta=\frac{1}{2L} ~\mbox{and}~
    \gamma_t=\min\left\{\frac{\epsilon}{\|\tilde{x}^{t+1}-x^t\|},\,1\right\}.
\end{equation}
Several variants of \textsc{Spider}-SFO, including ones that use constant 
step sizes and direct extensions of the form
$x^{t+1} = \prox_\Psi^\eta(x^t-\eta v^t)$,
have been proposed in the literature
\cite{SARAH2017ICML,SPIDER2018NeurIPS,SmoothSARAH2019,ProxSARAH2019}.
Algorithm~\ref{alg:NPAG} is very similar to ProxSARAH
\cite{ProxSARAH2019}, but with different choices of $\gamma_t$.
In ProxSARAH, the whole sequence $\{\gamma_t\}_{t\geq 0}$ is specified 
off-line in terms of parameters such as~$L$ and number of steps to run.  
Our choice of $\gamma_t$ follows closely that of \textsc{Spider}-SFO 
in~\eqref{eqn:spider-sfo}, which is automatically adjusted in an online 
fashion, and ensures bounded step length.
This property allows us to extend the algorithm further to work with more 
general gradient estimators, including those for multi-level 
composite stochastic optimization, all within the common framework of NPAG
and Theorem~\ref{theorem:cvg-SNPAG}.

The sample complexity of Algorithm~\ref{alg:prox-spider} is given in the 
following corollary.

\begin{corollary}\label{cor:prox-spider-stochastic}
Consider problem~\eqref{eqn:min-smooth+convex} with $F(x)=\E_\xi[f_\xi(x)]$.
Suppose Assumption~\ref{assumption:Lip-F} holds and 
the gradient mapping $\phi_\xi\equiv f'_\xi$ 
satisfies~\eqref{eqn:MSL-mapping} and~\eqref{eqn:phi-xi-var}
on $\dom\Psi$ (instead of $\R^d$).
If the parameters in Algorithm~\ref{alg:prox-spider} are chosen as
\begin{equation}\label{eqn:prox-spider-params}
    \eta = \frac{1}{2L}, \qquad
    \tau = \left\lceil \frac{2\sigma}{\epsilon}\right\rceil, \qquad
    B = \left\lceil \frac{2\sigma^2}{\epsilon^2}\right\rceil, \qquad
    b = \left\lceil \frac{\sigma}{\epsilon}\right\rceil ,
\end{equation}
where $\lceil \cdot\rceil $ denotes the integer ceiling,
then the output $\bar{x}$ satisfies $\E[\|\cG(\bar{x})\|]\leq\epsilon$ after 
$O(L\epsilon^{-2})$ iterations, and the total sample complexity is 
$O(L\sigma\epsilon^{-3}+\sigma^2\epsilon^{-2})$.
\end{corollary}
\begin{proof}
In Algorithm~\ref{alg:prox-spider}, the step lengths satisfy the same bound 
in~\eqref{eqn:npag-step-length}, \ie, 
$\|x^{t+1}-x^t\|\leq\eta\epsilon=\epsilon/2L$.
From the analysis following~\eqref{eqn:spider-batch-sizes}, the parameters
in~\eqref{eqn:prox-spider-params} guarantee~\eqref{eqn:vt-cond}.
Therefore we can apply Theorem~\ref{theorem:cvg-SNPAG} and Remark \ref{remark:constant-epsilon} to conclude that 
$\E[\|\cG(\bar{x})\|]\leq 5\epsilon$ 
provided that $T\geq 4L(\Phi(x^0)-\Phi_*)\epsilon^{-2}$.

To estimate the sample complexity, we can simply multiply the number of samples
required by simple mini-batching, $T\cdot\sigma^2/\epsilon^2$, by the optimal
ratio in~\eqref{eqn:opt-batch-ratio} to obtain
\[
    T\cdot \frac{\sigma^2}{\epsilon^2}\cdot \rho(\tau_\star)
    =\frac{4L(\Phi(x^0)-\Phi_*)}{\epsilon^2} \cdot \frac{\sigma^2}{\epsilon^2}
    \cdot \frac{2\,\epsilon}{\sigma}
    = O\left(\frac{L\sigma}{\epsilon^3}\right),
\]
where we used the assumption $\epsilon<\sigma$.
Considering that $B=O(\sigma^2/\epsilon^2)$ samples are always needed at $t=0$,
we can include it in the total sample complexity, which becomes
$O(L\sigma\epsilon^{-3}+\sigma^2\epsilon^{-2})$.
\end{proof}

We note that for each $t$ that is not a multiple of~$\tau$, 
we need to evaluate $f'_{\xi}$ twice for each $\xi\in\cB_t$, one at $x^t$ 
and the other at $x^{t-1}$. 
Thus the computational cost per iteration, measured by the number of evaluations of $f_\xi$ and $f'_{\xi}$, is twice the number of samples.
Therefore the overall computational cost is proportional 
to the sample complexity.
This holds for all of our results in this paper.

\subsection{SARAH/SPIDER estimator for finite-sum optimization}

Now consider a finite number of mappings 
$\phi_i:\R^d\to\R^{p\times q}$ for $i=1,\ldots,n$, and define
$\bar{\phi}(x)=(1/n)\sum_{i=1}^n \phi_i(x)$ for every~$x\in\R^d$.
We assume that the MSL condition~\eqref{eqn:MSL-mapping} holds,
but do \emph{not} require a uniform bound on the variance
as in~\eqref{eqn:phi-xi-var}.
Again, we would like to consider the number of samples required to construct
$\tau$ consecutive estimates $\{v^0,\ldots,v^{\tau-1}\}$ that satisfy
condition~\eqref{eqn:unif-var-bound}.
Without the uniform variance bound~\eqref{eqn:phi-xi-var}, this cannot
be treated exactly as a special case of what we did in 
Section~\ref{sec:spider-expect}, and thus deserves separate attention.

We use the same \textsc{Sarah}/\textsc{Spider} estimator~\eqref{eqn:spider} 
and choose $\cB_0=\{1,\ldots,n\}$.
In this case, we have $v^0=\bar{\phi}(x^0)$ and 
$\E\bigl[\|v^0-\bar{\phi}(x^0)\|^2\bigr]=0$.
From~\eqref{eqn:spider-var-bound}, if we can control the step lengths
$\|x^t-x^{t-1}\|\leq\epsilon/2L$ and use a constant batch size $\cB_t=b$ 
for $t=1,\ldots,\tau-1$, then
\[
    \E\left[\bigl\|v^1-\bar\phi(x^1)\bigr\|^2\right] 
    ~\leq~\cdots~\leq~
    \E\left[\bigl\|v^{\tau-1}-\bar\phi(x^{\tau-1})\bigr\|^2\right] 
    ~\leq~ \sum_{t=1}^{\tau-1} \frac{\epsilon^2}{4|\cB_t|} 
    ~\leq~ \frac{\tau\epsilon^2}{4b}.
\]
In order to satisfy~\eqref{eqn:unif-var-bound}, it suffices to use
$b=\tau/4$. The number of samples required for~$\tau$ steps is
\[
    |\cB_0| + (\tau-1)b ~\leq~ n + \tau^2/4.
\]
As before, we compare the above number of samples against
the simple mini-batching scheme of using 
$v^t=\phi_{\cB_t}(x^t)$ for all $t=0,\ldots,\tau-1$. 
However, without an assumption like~\eqref{eqn:phi-xi-var}, we cannot
guarantee~\eqref{eqn:unif-var-bound} unless using $\cB_t=\{1,\ldots,n\}$ for
all~$t$ in simple mini-batching, which becomes full-batching.
The ratio between the numbers of samples required by \textsc{Spider}
and full-batching is
\[
    \rho(\tau) = \frac{n+\tau^2/4}{n\tau} = \frac{1}{\tau} + \frac{\tau}{4n} .
\]
By minimizing this ratio, we obtain the optimal epoch length and batch size:
\[
  \tau_\star=\sqrt{2n},\qquad b_\star = \sqrt{2n}/4.
\]
In this case, the minimum ratio is $\rho(\tau_\star)= 1/\sqrt{n}$.
Hence we expect significant saving in sample complexity for large~$n$.
Similar to the expectation case, we should restart the  
estimator every $\tau_\star=\sqrt{2n}$ iterations to keep the optimal 
ratio of reduction.

As an example, we consider using the Prox-\textsc{Spider} method to solve
problem~\eqref{eqn:min-smooth+convex} when the smooth part of the objective 
is a finite-sum. 
The resulting sample complexity is summarized in the following corollary of 
Theorem~\ref{theorem:cvg-SNPAG}.
Its proof is similar to that of Corollary~\ref{cor:prox-spider-stochastic}
and thus omitted.

\begin{corollary}\label{cor:prox-spider-finite-sum}
Consider problem~\eqref{eqn:min-smooth+convex} with 
$F(x)=(1/n)\sum_{i=1}^n f_i(x)$.
Suppose Assumption~\ref{assumption:Lip-F} holds and the gradient $f'_i$ 
satisfies~\eqref{eqn:MSL-mapping} on $\dom\Psi$.
If we set the parameters in Algorithm~\ref{alg:prox-spider} as
\[
    \eta = 1/2L, \qquad
    \tau = \left\lceil\sqrt{2n}\right\rceil, \qquad
    B = n, \qquad
    b = \left\lceil\sqrt{2n}/4\right\rceil,
\]
and use $\cB_t=\{1,\ldots,n\}$ whenever $t\!\!\mod\!\tau=0$, 
then the output satisfies 
$\E[\|\cG(\bar{x})\|]\leq\epsilon$ after $O(L\epsilon^{-2})$ iterations,
and the total sample complexity is $O\bigl(n+L\sqrt{n}\epsilon^{-2}\bigr)$.
\end{corollary}

\subsection{SVRG estimator for stochastic optimization}
\label{subsec:SVRG}
As a general framework, NPAG may also incorporate the SVRG estimator proposed in \cite{johnson2013accelerating}. We illustrate how SVRG estimator can be applied to construct $\tau$ consecutive estimates $\{v^0,...,v^{\tau-1}\}$ that satisfy \eqref{eqn:unif-var-bound} for stochastic optimization (the expectation case). 
For the SVRG estimator, we have 
\begin{equation}\label{eqn:svrg}
\begin{split}
v^0 &= \phi_{\cB_0}(x^0), \\
v^t &= v^{0}  + \phi_{\cB_t}(x^t) - \phi_{\cB_t}(x^0).
\end{split}
\end{equation} 
In this case, we have (c.f.\ Lemma~\ref{lem:spider})
\[
\E\bigl[\bigl\|v^t-\bar\phi(x^t)\bigr\|^2\bigr] 
\leq \E\bigl[\bigl\|v^0-\bar\phi(x^0)\bigr\|^2\bigr] + \frac{L^2}{|\cB_t|}\E\left[\|x^t-x^0\|^2\right],
\]
see \cite[Lemma~4]{Reddi2016SVRGnonconvex}. 
If we bound the step length by~$\delta$, then by the triangle inequality,
$$\|x^t-x^0\| \leq \sum_{r=1}^t\|x^r-x^{r-1}\| \leq t\delta,$$ 
which gives
\[
\E\left[\bigl\|v^t-\bar\phi(x^t)\bigr\|^2\right] 
\leq \frac{\sigma^2}{|\cB_0|} + \frac{t^2 L^2\delta^2}{|\cB_t|}.
\]
To guarantee~\eqref{eqn:unif-var-bound}, we can choose $\delta = \epsilon/2L$ and
$|\cB_0|\geq 2\sigma^2/\epsilon^2$ and for $t=1,\ldots,\tau$, and 
$|\cB_t|=b\geq 2 \tau^2 L^2\delta^2/\epsilon^2=\tau^2/2$.
Thus the total sample complexity per epoch is
\[
|\cB_0| + (\tau-1)b = 
\frac{2\sigma^2}{\epsilon^2} + (\tau-1)\frac{\tau^2}{2}
\leq \frac{2\sigma^2}{\epsilon^2} + \frac{\tau^3}{2}.
\]
We can minimize the ratio between the above quantity and $\tau\sigma^2/\epsilon^2$ (of the naive mini-batch scheme) by choosing $\tau=(\sqrt{2}\sigma/\epsilon)^{2/3}$ and total sample complexity per epoch becomes
$O(\sigma^2/\epsilon^2)$, which can be much smaller than 
$\tau\sigma^2/\epsilon^2$. 
By restarting the gradient estimation every $\tau$ iterations, \eqref{eqn:vt-cond} is satisfied with $\epsilon_t\equiv\epsilon$. By Theorem \ref{theorem:cvg-SNPAG} and Remark \ref{remark:constant-epsilon}, getting a $\bar x$ s.t. $\E\left[\|\cG(\bar x)\|\right]\leq5\epsilon$ requires $T = O(L\epsilon^{-2})$ iterations and a total complexity of $O\big((T/\tau)\cdot(\sigma^2/\epsilon^2)\big) = O(L\sigma^{4/3}\epsilon^{-10/3})$, which is slightly worse than that of the SARAH/\textsc{Spider} estimator in the expectation case.

The finite-sum case is similar and is hence omitted. 
It will have the same sample complexity as the SAGA estimator as we discuss next.

\subsection{SAGA estimator for finite-sum optimization}
\label{subsec:SAGA}
In this section, we discussion how the SAGA estimator proposed in \cite{defazio2014saga} can be incorporated into our NPAG framework. Note that the SAGA estimator is only meaningful for the finite-sum case (sampling with replacement). 
The SAGA estimator can be written as
\begin{eqnarray*}
	v^0 &=& u^0 ~=~ \frac{1}{n}\sum_{i=1}^n \phi_i(\alpha_i^0)\quad\mbox{with} \quad\alpha_i^0 = x^0,~ i=1,\ldots,n, \\
	v^t &=& u^{t-1} + \frac{1}{|\cB_t|}\sum_{i\in\cB_t}\bigl(\phi_i(x^t)-\phi_i(\alpha_i^{t-1})\bigr), \\
	u^t &=& u^{t-1} + \frac{1}{n}\sum_{i\in\cB_t}\bigl(\phi_i(x^t)-\phi_i(\alpha_i^{t-1})\bigr).
\end{eqnarray*}
After each iteration, we update $\alpha_i^t$ by setting $\alpha_i^t\leftarrow x^t$ if $i\in\cB_t$ and $\alpha_i^t\leftarrow\alpha_i^{t-1}$ if $i\notin\cB_t$. Denote $\alpha[n,t] = \{\alpha_i^r\}_{i=1,\ldots,n}^{r=0,\ldots,t}$. Suppose $|\cB_t|\equiv b$ for all $t$. Then
the MSE of the SAGA estimator is bounded by
\begin{eqnarray*}
	\E\bigl[\|v^t-\bar{\phi}(x^t)\|^2 | \alpha[n,t]\bigr] \leq \frac{L^2}{nb}\sum_{i=1}^n\|x^t-\alpha_i^{t-1}\|^2.
\end{eqnarray*}
Note that the $\alpha_i^{t-1}$'s are identically distributed, thus we further have 
\begin{equation}
\label{eqn:new-2}
\E\bigl[\|v^t-\bar{\phi}(x^t)\|^2 \bigr] \leq \frac{L^2}{b}\E\left[\|x^t-\alpha_1^{t-1}\|^2\right].
\end{equation}
When we restrict the step lengths by $\|x^r - x^{r-1}\|\leq \delta$, then we have $\|x^t-x^r\|\leq |t-r|\delta$ for all $t,r$. 
Note that all the batches $\cB_r$ are independently and uniformly sampled and $\alpha_1^{t-1} = x^{t-1}$ if $1\in\cB_{t-1}$. That is, $\alpha_1^{t-1} = x^{t-1}$ with probability $b/n$. Similarly, for $r<t-1$, we have $\alpha_1^{t-1} = x^{r}$ if $1\in\cB_{r}$ and $1\notin \cup_{j = r+1}^{t-1}\cB_{j}$. Hence $\alpha_1^{t-1} = x^{r}$ with probability $(b/n)(1-b/n)^{t-r-1}$. Therefore, \eqref{eqn:new-2} implies
\begin{eqnarray}
\label{eqn:vt-saga}
\E\bigl[\|v^t-\bar{\phi}(x^t)\|^2 \bigr] & \leq &  \frac{L^2}{b}\cdot \sum_{r=0}^{t-1}\frac{b}{n}\big(1-\frac{b}{n}\big)^{t-r-1}\E\left[\|x^t-x^r\|^2\right]\\
&\leq& \frac{L^2}{n}\cdot \sum_{r=0}^{t-1}\big(1-\frac{b}{n}\big)^{t-r-1}(t-r)^2\delta^2\nonumber\\
&\leq& \frac{3n^2L^2\delta^2}{b^3}.\nonumber
\end{eqnarray}
When we choose $\delta = \epsilon/2L$ and $b = n^{2/3}$, we have $\E\bigl[\|v^t-\bar{\phi}(x^t)\|^2 \bigr]\leq\epsilon^2$ for any $t\geq0$. Consequently, Theorem \ref{theorem:cvg-SNPAG} and Remark \ref{remark:constant-epsilon} implies a total sample complexity of $O(n^{2/3}\epsilon^{-2})$ to output a point $\bar x$ s.t. $\E\big[\|\cG(\bar{x})\|\big]\leq 5\epsilon$.
This sample complexity is worse than that of the SARAH/\textsc{Spider} estimator.

\smallskip

In this section, we analyzed the normalized-step SARAH/\textsc{Spider}, SVRG and SAGA estimators. Specifically, we described SARAH/\textsc{Spider} and SVRG from the perspective of efficient variance reduction over a period of consecutive iterations in a stochastic algorithm. This perspective allows us to derive the optimal period length and mini-batch sizes of the estimator without binding with any particular algorithm. Using these estimators in the NPAG framework led to the Prox-\textsc{Spider}
method, and we obtained its sample complexity by simply combining the iteration complexity of NPAG (Theorem~\ref{theorem:cvg-SNPAG})
and the optimal period and mini-batch sizes of SARAH/\textsc{Spider}. The complexities in these cases on one-level stochastic optimization are not new (see, \eg, \cite{SPIDER2018NeurIPS,SpiderBoost2018,ProxSARAH2019,Reddi2016SVRGnonconvex,NCVX-SAGA-Smooth}). 
But the NPAG framework and the separate perspective on variance 
reduction are more powerful and also work for multi-level stochastic optimization, which we show next.

\section{Multi-level nested SPIDER}
\label{sec:multi-level}

In this section, we present results for the general composite optimization
problems~\eqref{eqn:multi-level-stochastic} 
and~\eqref{eqn:multi-level-finite-sum} with $m\geq 2$. To provide convergence and complexity result, we only need to construct a gradient estimator 
$v^t$ that is sufficiently close to $F'(x^t)$.
For the ease of presentation, let us denote 
\[
    F_k = f_k\circ f_{k-1}\circ\cdots\circ f_1, \qquad k=1,\ldots,m.
\] 
In particular, we have $F_1 = f_1$, and $F_m = F$.
Then, by chain rule,
\begin{equation}\label{eqn:m-level-chain-rule}
F'(x) = \bigl[f_1'(x)\bigr]^T\bigl[f_2'(F_1(x))\bigr]^T\cdots \bigl[f'_{m-1}(F_{m-2}(x))\bigr]^T f'_m(F_{m-1}(x)).
\end{equation}
To approximate $F'(x^t)$, we construct $m-1$ estimators for the mappings,
\begin{equation}\label{eqn:m-1-mapping-estimators}
y_1^t\approx F_1(x^t), \quad \ldots, \quad y_{m-1}^t\approx F_{m-1}(x^t),
\end{equation}
and $m$ estimators for the Jacobians,
\begin{equation}\label{eqn:m-Jacobian-estimators}
z_1^t\approx f'_1(x^t),\quad z_2^t\approx f'_2(y_1^t), \quad \ldots, \quad z_m^t\approx f'_m(y_{m-1}^t).
\end{equation}
(Notice that we do not need to estimate $F(x^t)=F_m(x^t)$.)
Then we can construct  
\begin{equation}\label{eqn:m-level-estimator}
v^t = (z_1^t)^T(z_2^t)^T\cdots(z_{m-1}^t)^T z_m^t.
\end{equation}
We use the SARAH/\textsc{Spider} estimator~\eqref{eqn:spider}
for all $m-1$ mappings in~\eqref{eqn:m-1-mapping-estimators} 
and $m$ Jacobians in~\eqref{eqn:m-Jacobian-estimators},
then use the above $v^t$ as the approximate gradient in the NPAG method.
This leads to the multi-level Nested-\textsc{Spider} method shown
in Algorithm~\ref{alg:nested-spider-m}.
To simplify presentation, we adopted the notation $y_0^t=x^t$ for all $t\geq 0$.

\begin{algorithm2e}[t]
	\DontPrintSemicolon
    \caption{Multi-level Nested-\textsc{Spider} Method}
	\label{alg:nested-spider-m}
	\textbf{input:} 
	initial point $x_0\in\R^d$, proximal step size $\eta>0$, a sequence of precisions\\
	\hspace{8ex}$\{\epsilon_k\}_{k=1}^K$, 
    stage lengths~$\{\tau_k\}_{k=1}^K$, and batch sizes $\{B_i^k,S_i^k,b_i^k,s_i^k\}_{i=1,\ldots,m}^{k=1,\ldots,K}$. \\
    \textbf{initialize:} $y_0^0 = x^0$, $k=1$, $s = 0$\\ 
	\For{$t = 0, 1,...,T-1$}{ 
		$\chi(t) = k$. \,\,\,\,\,\,\,\,\,\texttt{/* iteration $t$ is in the $k$-th stage */}\\
        \eIf{$s ==0$}{
            \For{$i=1,\ldots,m$}{
                Randomly sample batch $\cS_i^t$ of $\xi_i$ with $|\cS_i^t|\!=\!S_i^k$, batch $\cB_i^t$\! with\! $|\cB_i^t|\!=\!B_i^k$\!\!\!\!\!\!\\
                $y_i^t=f_{i,\cS_i^t}(y_{i-1}^t)$  
                \qquad \texttt{/* skip if $i==m$ */}\\ 
                $z_i^t=f'_{i,\cB_i^t}(y_{i-1}^t)$ \\
            }
        }{
            \For{$i=1,\ldots,m$}{
                Randomly sample batch $\cS_i^t$ of~$\xi_i$ with $|\cS_i^t|\!=\!s_i^k$, batch $\cB_i^t$ with $|\cB_i^t|\!=\!b_i^k$\!\!\!\!\!\!\\
                $y_i^t=y_i^{t-1}+f_{i,\cS_i^t}(y_{i-1}^t)-f'_{i,\cS_i^t}(y_{i-1}^{t-1})$ 
                \qquad \texttt{/* skip if $i==m$ */}\\ 
                $z_i^t=z_i^{t-1}+f'_{i,\cB_i^t}(y_{i-1}^t)-f'_{i,\cB_i^t}(y_{i-1}^{t-1})$ \\
            }
        }
        $v^t = (z_1^t)^T(z_2^t)^T\cdots(z_{m-1}^t)^T z_m^t$\\[0.5ex]
        $\tilde{x}^{t+1} = \mathbf{prox}_\Psi^\eta (x^t-\eta v^t)$\\
        $x^{t+1} = x^t  + \gamma_t(\tilde{x}^{t+1} - x^t)$ ~with~
$\gamma_t=\min\left\{\frac{\eta\,\epsilon_{\chi(t)}}{\|\tilde{x}^{t+1}-x^t\|},~1\right\}$\\
    $y_0^{t+1} = x^{t+1}$, $s \leftarrow s+1$\\
    \textbf{if}\, $s==\tau_k$\,\,\textbf{then}\, $s \leftarrow 0$, $k \leftarrow k+1$ \quad \texttt{/* reset $s=0$ for next stage */}
    }
    \textbf{output:} 
    $\bar{x}$ from $\{x_0,...,x^{T-1}\}$ with $x^t$ sampled with probability $\frac{\epsilon_{\chi(t)}}{\sum_{k=1}^K\tau_k\epsilon_k}$.
\end{algorithm2e} 

Each iteration~$t$ in Algorithm~\ref{alg:nested-spider-m} is assigned a stage or epoch indicator $\chi(t)\in\{1,\ldots,K\}$, indicating the current variance-reduction stage.
Each stage~$k$ has length~$\tau_k$ which is determined by the variable precision~$\epsilon_k$.
Note that all iterations within the same stage~$k$ share the same $\epsilon_{\chi(t)}=\epsilon_k$, which in turn is used to determine the step size
$\gamma_t=\min\left\{\eta\,\epsilon_{\chi(t)}/\|\tilde{x}^{t+1}-x^t\|,~1\right\}$.

For convergence analysis, we make the following assumption. 
\begin{assumption} \label{assumption:m=m} 
    For the functions and mappings appearing in~\eqref{eqn:multi-level-stochastic} and~\eqref{eqn:multi-level-finite-sum}, we assume:
\begin{itemize} \itemsep 0pt
    \item[(a)] For each $i=1,\ldots,m$ and each realization of $\xi_i$, mapping $f_{i,\xi_i}\!:\R^{d_{i-1}}\!\to\!\R^{d_i}$ is $\ell_i$-Lipschitz and its Jacobian $f'_{i,\xi_i}\!:\R^{d_{i-1}}\!\to\!\R^{d_i\times d_{i-1}}$ is $L_i$-Lipschitz. 
    \item[(b)] For each $i=1,\ldots,m-1$, there exists $\delta_i$ such that
$\E_{\xi_i}\bigl[\|f_{i,\xi_i}(y)-f_i(y)\|^2\bigr]\leq \delta_i^2$ for any $y\in\dom f_i$.
    \item[(c)] For each $i=1,\ldots,m$, there exists $\sigma_i$ such that 
$\E_{\xi_i}\bigl[\|f'_{i,\xi_i}(y)-f'_i(y)\|^2\bigr]\leq \sigma_i^2$ for any $y\in\dom f_i$.
\end{itemize}
\end{assumption}
As a result, the function $F = f_m\circ\cdots\circ f_1$ 
and its gradient $F'$ are both Lipschitz continuous, and their Lipschitz 
constants are given in the following lemma.
\begin{lemma}\label{lem:m=m-Lip-constants}
    Suppose Assumption~\ref{assumption:m=m}.(a) holds.
    Then the composite function~$F$ and its gradient~$F'$ are Lipschitz 
    continuous, with respective Lipschitz constants
\begin{eqnarray}
\label{eqn:Lip-F-m=m}
\ell_F=\sprod_{i=1}^m \ell_i, \qquad
L_F = \sum_{i=1}^m L_i \left(\sprod_{r=1}^{i-1}\ell_r^2\right)\left(\sprod_{r=i+1}^{m}\ell_r\right),
\end{eqnarray}
where we use the convention $\sprod_{r=1}^0\ell_r = 1$. 
\end{lemma}
The proof of this lemma is presented in Appendix~\ref{sec:Lip-F-m-proof}.
For the special case $m=2$, 
\[
\ell_F=\ell_1\ell_2, \qquad L_F = \ell_2 L_1 + \ell_1^2 L_2,
\]
which have appeared in previous work on two-level problems 
\cite{GhadimiRuszWang2018,ZhangXiao2019C-SAGA}.

In addition, we define two constants:
\begin{equation}\label{eqn:sigma-F-delta-F}
    \sigma_F^2 = \sum_{i=1}^m\biggl(\sprod_{r\neq i}\ell_r^2\biggr)\sigma_i^2,
    \qquad
    \delta_F^2 = \sum_{i=1}^m\frac{L_F^2}{\sprod_{r=1}^i\ell_r^2}\delta_i^2 .
\end{equation}
They will be convenient in characterizing the sample complexity of 
Algorithm~\ref{alg:nested-spider-m}.

As inputs to Algorithm~\ref{alg:nested-spider-m}, we need to choose the 
variable precisions $\epsilon_k$, epoch length~$\tau_k$ and batch sizes 
$\{B_i^k,S_i^k,b_i^k,s_i^k\}_{i=1,\ldots,m}^{k=1,\ldots,K}$ to ensure
$\E[\|v^t-F'(x^t)|\|^2]\leq\epsilon_{\chi(t)}^2$.
This is done through the following lemma.

\begin{lemma}\label{lemma:m=m-v-mse}
Suppose Assumption~\ref{assumption:m=m} holds.
In Algorithm~\ref{alg:nested-spider-m}, if we set $\eta=\frac{1}{2L_F}$, $\tau_k = \frac{\ell_F}{2m\epsilon_k}$ and
\begin{eqnarray}
&& B_i^k =\frac{12m(m+1)\sigma_F^2}{\epsilon_k^2}, \quad
    b_i^k 
    = \frac{6(m+1)\ell_F}{\epsilon_k}, \quad~~ i=1,\ldots,m,
    \label{eqn:m=m-B-b} \\
&& S_i^k = \frac{12m(m+1)\delta_F^2}{\epsilon^2_k}, \quad 
    s_i^k 
    = \frac{6m(m+1)\ell_F}{\epsilon_k}, \quad i=1,\ldots,m-1,\qquad
    \label{eqn:m=m-S-s}
\end{eqnarray}
then $\E\bigl[\|v^t-F'(x^t)\|^2\bigr] \leq \epsilon^2_k$ for all $k$ and all $t\in\{r:\chi(r) = k\}$.  
\end{lemma}

\begin{proof}
First, we bound $\|v^t-F'(x^t)\|^2$ in terms of the approximation
errors of the individual estimators in~\eqref{eqn:m-1-mapping-estimators}
and~\eqref{eqn:m-Jacobian-estimators}.
Denoting $F_i(x^t)$ by $F_i^t$ for simplicity, we have 
from~\eqref{eqn:m-level-chain-rule} and~\eqref{eqn:m-level-estimator},
\[
\E\bigl[\|v^t\!-F'(x^t)\|^2\bigr] 
\!= \E\Bigl[\left\|[z_1^t]^T[z_2^t]^T\!\cdots[z_{m-1}^t]^T\! z_m^t 
-[f_1'(x^t)]^T[f'_2(F_1^t)]^T\!\cdots f'_m(F_{m-1}^t)\right\|^2\Bigr].
\]
To bound this value, we add and subtract intermediate terms inside
the norm such that each adjacent pair of products differ at most in one factor.
Then we use the inequality
$\bigl\|\sum_{i=1}^r a_i-b_i\bigr\|^2\leq r\sum_{i=1}^r\|a_i-b_i\|^2$ 
to split the bound. 
More concretely, we split $v^t-F'(x^t)$ into $k=2m-1$ pairs:
\begin{eqnarray}
\label{eqn:m=m-v-mse-1}
&& \qquad \E\bigl[\|v^t-F'(x^t)\|^2\bigr] \\
& \leq & (2m-1)\biggl(\E\left[\big\|[z_1^t]^T[z_2^t]^T\cdots[z_{m-1}^t]^Tz_m^t - [f_1'(x^t)]^T[z_2^t]^T\cdots[z_{m-1}^t]^Tz_m^t \big\|^2\right] \nonumber\\
& & + \E\!\left[\big\|[f_1'(x^t)]^T[z_2^t]^T[z_3^t]^T\!\!\cdots[z_{m-1}^t]^Tz_m^t  - [f_1'(x^t)]^T[f'_2(y_1^t)]^T[z_3^t]^T\!\!\cdots[z_{m-1}^t]^Tz_m^t\big\|^2\right] \nonumber \\
& & + \E\!\left[\big\|[f_1'(x^t)]^T[f'_2(y_1^t)]^T[z_3^t]^T\!\!\cdots\![z_{m\!-\!1}^t]^T\! z_m^t  \!-\! [f_1'(x^t)]^T[f'_2(F_1^t)]^T[z_3^t]^T\!\!\cdots\![z_{m\!-\!1}^t]^Tz_m^t\big\|^2\right] \nonumber \\[1ex]
& & + ~\cdots \nonumber \\[1ex]
& & + \E\!\left[\big\|[f_1'(x^t)]^T\cdots[f'_{m-1}(F_{m-2}^t)]^Tz_m^t  - [f_1'(x^t)]^T\cdots[f'_{m-1}(F_{m-2}^t)]^Tf'_m(y_{m-1}^t)\big\|^2\right] \nonumber \\
& & + \E\!\left[\big\|[f_1'(x^t)]^T\!\!\cdots\![f'_{m\!-\!1}(F_{m\!-\!2}^t)]^T\! f'_m(y_{m\!-\!1}^t)\!-\![f_1'(x^t)]^T\!\!\cdots\![f'_{m\!-\!1}(F_{m\!-\!2}^t)]^T\! f'_m(F_{m\!-\!1}^t)\big\|^2\right]\!\biggr) \nonumber \\
& \leq & 2m\sum_{i=1}^m\E\biggl[\biggl(\sprod_{r=1}^{i-1}\|f'_r(F_{r-1}^t)\|^2\biggr)\biggl(\sprod_{r=i+1}^m \|z_r^t\|^2\biggr) \nonumber \\[-1ex]
&& \qquad\qquad\qquad \biggl(\|z_i^t-f_i'(y_{i-1}^t)\|^2 + \mathbf{1}_{i\geq2}\|f'_i(y_{i-1}^t)-f'_i(F_{i-1}^t)\|^2\biggr)\biggr]  \nonumber \\
& \leq & 2m\sum_{i=1}^m\E\biggl[\biggl(\sprod_{r=1}^{i-1}\ell_r^2\biggr)\biggl(\sprod_{r=i+1}^m \|z_r^t\|^2\biggr)\biggl(\|z_i^t-f_i'(y_{i-1}^t)\|^2 + \mathbf{1}_{i\geq2}L_i^2\|y_{i-1}^t-F_{i-1}^t\|^2\biggr)\biggr],
\nonumber 
\end{eqnarray}
where we used the notation $F_0^t=y_0^t=x^t$ and $\mathbf{1}_{i\geq 2}=1$
if $i\geq 2$ and~$0$ otherwise.
In the last inequality, we used $\|f'_r(\cdot)\|\leq\ell_r$, which is a 
consequence of $f_r$ being $\ell_r$-Lipschitz.

Next, we first derive deterministic bounds on $\|z_i^t\|^2$ so that they
can be moved outside of the expectation in~\eqref{eqn:m=m-v-mse-1}. 
Then we will bound $\E\bigl[\|z_i^t-f_i'(y_{i-1}^t)\|^2\bigr]$ and
$\E\bigl[\|y_{i-1}^t-F_{i-1}^t\|^2\bigr]$ separately.
Without loss of generality, we focus on the first epoch with $t=0,1,\ldots,\tau_1-1$, but the results hold for all the epochs. For the ease of notation, we denote $\epsilon_1$, $\tau_1$ and $B_i^1,S_i^1,b_i^1,s_i^1$ by $\epsilon$, $\tau$ and $B_i,S_i,b_i,s_i$ in this proof. 

\emph{Step 1: Bounding the temporal differences $\|y_i^t-y_i^{t-1}\|$.}
For $i=0$, we have $y_0^t=x^t$ and from~\eqref{eqn:npag-step-length}, 
$\|x^t-x^{t-1}\|\leq \eta\epsilon=\epsilon/2L_F\leq\epsilon/L_F$.
When $i=1$, for any $t\geq 1$, 
\begin{eqnarray}
	\|y_1^t-y_1^{t-1}\| \leq \bigg\|\frac{1}{s_1}\sum_{\xi_1\in\cS^r_1}\bigl(f_{1,\xi_1}(x^t)-f_{1,\xi_1}(x^{t-1})\bigr)\bigg\|\leq \ell_1\|x^t-x^{t-1}\|\leq \frac{\ell_1\epsilon}{L_F}.\nonumber
\end{eqnarray}
In general, for all $1\leq i\leq m-1$, we have 
\begin{eqnarray}
\label{eqn:m=m-y-diff.}
\qquad\quad \|y_i^t\!-\!y_i^{t-\!1}\| 
\leq \bigg\|\frac{1}{s_i}\!\!\sum_{\xi_i\in\cS^t_i}\!\!\bigl(f_{i,\xi_i}\!(y_{i-1}^t)\!-\!f_{i,\xi_i}\!(y_{i-1}^{t-1})\bigr)\bigg\|
\leq \ell_i\|y_{i-1}^t\!-\!y_{i-1}^{t-1}\|\leq \frac{\prod_{r=1}^i\ell_r}{L_F}\epsilon.
\end{eqnarray}

\emph{Step 2: Bounding $\|z_i^t\|$.}
We have $\|z_i^t\|\leq \|z^0_i\| + \sum_{r=1}^{t}\|z_i^r-z_i^{r-1}\|$ 
for all $1\leq i\leq m$. Moreover,
\[
\|z_i^r\!-\!z_i^{r-1}\| \leq \bigg\|\frac{1}{b_i}\!\sum_{\xi_i\in\cB^t_i}\!\bigl(f'_{i,\xi_i}(y_{i-1}^r)-f'_{i,\xi_i}(y_{i-1}^{r-1})\bigr)\bigg\|\leq L_i\|y_{i-1}^r-y_{i-1}^{r-1}\|\leq L_i\!\left(\sprod_{r=1}^{i-1}\ell_r\!\right)\!\frac{\epsilon}{L_F},
\]
where the last inequality is due to~\eqref{eqn:m=m-y-diff.}.
Consequently, for all $0\leq t\leq \tau-1$,
\begin{equation}
\label{eqn:m=m-z-bound}
\|z_i^t\|\leq \|z_i^0\| + t\cdot\left(\sprod_{r=1}^{i-1}\ell_r\right)\frac{L_i\epsilon}{L_F}\leq \ell_i + \tau\cdot L_i\left(\sprod_{r=1}^{i-1}\ell_r\right)\frac{\epsilon}{L_F}. 
\end{equation} 

\emph{Step 3: Bounding $\E\bigl[\|z^t_i-f'_i(y^t_{i-1})\|^2\bigr]$ and
$\E\bigl[\|y^t_i-f_i(y^t_{i-1})\|^2\bigr]$.}
By Lemma \ref{lem:spider}, we get 
\begin{eqnarray}
\E\bigl[\|z^t_i-f'_i(y^t_{i-1})\|^2\bigr]  & \leq & \E\bigl[\|z_i^0-f_i'(y_{i-1}^0)\|^2\bigr] + \frac{L_i^2}{b_i}\sum_{r=1}^t\E\bigl[\|y^r_{i-1} - y^{r-1}_{i-1}\|^2\bigr] \nonumber \\
& \leq & \frac{\sigma_i^2}{B_i} + {L_i^2}\left(\sprod_{r=1}^{i-1}\ell_r^2\right)\frac{\tau\epsilon^2}{b_iL^2_F}.
\label{eqn:m=m-z-mse}
\end{eqnarray}
Similarly, using Lemma~\ref{lem:spider} again, we obtain
\begin{eqnarray}
\label{eqn:m=m-y-mse}
\E\bigl[\|y^t_i-f_i(y^t_{i-1})\|^2\bigr]  
& \leq & \frac{\delta_i^2}{S_i} + \left(\sprod_{r=1}^{i}\ell_r^2\right)\frac{\tau\epsilon^2}{s_iL^2_F}.
\end{eqnarray}

\emph{Step 4: Bounding $\E[\|y_i^t-F_i(x^t)\|^2]$.}
We show that the following inequality holds for $i=1,\ldots,m$:
\begin{equation}
\label{eqn:m=m-y-full}
\E\bigl[\|y_i^t-F_i(x^t)\|^2\bigr] ~\leq~ i\cdot\left[\sum_{r=1}^{i}\biggl(\sprod_{j=r+1}^i\!\!\ell_j^2\biggr)\frac{\delta_r^2}{S_r}\right] + i\cdot\biggl(\sprod_{j=1}^i\ell_j^2\biggr)\left[\sum_{r=1}^i\frac{\tau\epsilon^2}{s_rL_F^2}\right]. 
\end{equation}
(Here we use the convention $\sprod_{j=i+1}^i\ell_j=1$.) 
We prove this result by induction.
With the notations $F_1=f_1$ and $y_0^t=x^t$, 
the base case of $i=1$ is the same as \eqref{eqn:m=m-y-mse}, 
thus already proven.
Suppose \eqref{eqn:m=m-y-full} holds for $i=k$, for $i = k+1$, we have 
\begin{eqnarray*} 
& & \E\bigl[\|y_{k+1}^t-F_{k+1}(x^t)\|^2\bigr]\\
& \leq & (1+k)\E\bigl[\|y_{k+1}^t-f_{k+1}(y_{k}^t)\|^2\bigr] + (1+k^{-1})\E\bigl[\|f_{k+1}(y_{k}^t) - f_{k+1}(F_{k}(x^t))\|^2\bigr]\\
& \leq & (1+k)\E\bigl[\|y_{k+1}^t-f_{k+1}(y_{k}^t)\|^2\bigr] + (1+k^{-1})\ell_{k+1}^2\E\bigl[\|y_{k}^t - F_{k}(x^t)\|^2\bigr]\\
& \leq & (1+k)\cdot\left[\frac{\delta_{k+1}^2}{S_{k+1}} + \left(\sprod_{r=1}^{k+1}\ell_r^2\right)\frac{\tau\epsilon^2}{s_{k+1}L^2_F}\right] + (1+k^{-1})\ell_{k+1}^2\cdot k\cdot\left[\sum_{r=1}^{k}\biggl(\sprod_{j=r+1}^k\ell_j^2\biggr)\frac{\delta_r^2}{S_r}\right]\\
& & + (1+k^{-1})\ell_{k+1}^2\cdot k\cdot\left[\biggl(\sprod_{j=1}^k\ell_j^2\biggr)\sum_{r=1}^k\frac{\tau\epsilon^2}{s_rL_F^2}\right]\\
& = & (k+1)\cdot\left[\biggl(\sprod_{j=r+1}^{k+1}\ell_j^2\biggr)\sum_{r=1}^{k+1}\frac{\delta_r^2}{S_r}\right] + (k+1)\cdot\biggl(\sprod_{j=1}^{k+1}\ell_j^2\biggr)\left[\sum_{r=1}^{k+1}\frac{\tau\epsilon^2}{s_rL_F^2}\right],
\end{eqnarray*}
which completes the induction. Therefore~\eqref{eqn:m=m-y-full} holds.

\smallskip

Now we go back to inequality~\eqref{eqn:m=m-v-mse-1} and first apply the bound
on $\|z_i^t\|$ in~\eqref{eqn:m=m-z-bound}.
Using
\[
    \tau = \frac{\ell_F}{2m\epsilon} = \frac{\sprod_{j=1}^m \ell_j}{2m\epsilon},
\]
we have for $r=1,\ldots,m$ and all $t\geq 0$,
\[
\|z_r^t\| \leq
\ell_r + \tau L_r\sprod_{j=1}^{r-1}\ell_j\frac{\epsilon}{L_F} 
= \ell_r\Biggl(1+ \frac{L_r\left(\sprod_{j=1}^{r-1}\ell_j^2\right)\sprod_{j=r+1}^m\ell_j}{2mL_F}\Biggr)
\leq \left(1+\frac{1}{2m}\right)\ell_r,
\]
where the last inequality is due to the definition of $L_F$ 
in~\eqref{eqn:Lip-F-m=m}.
Consequently,  
\[
\biggl(\sprod_{r=1}^{i-1}\!\ell_r^2\!\biggr)\!\biggl(\sprod_{r=i+1}^m \!\!\|z_r^t\|^2\!\biggr) \leq
\sprod_{r=1}^{i-1}\!\ell_r^2 \!\sprod_{r=i+1}^m\!\!\biggl(\ell_r^2\biggl(1+\frac{1}{2m}\biggr)^{\!\!2}\biggr) \leq \biggl(\!1+\frac{1}{2m}\!\biggr)^{\!\!2m}\!\!\sprod_{r\neq i}\!\ell_r^2 \leq 3 \!\sprod_{r\neq i}\!\ell_r^2,
\]
where we applied the inequality $(1+a^{-1})^a\leq e\leq 3$ for $a>0$. 
Applying the above inequality to~\eqref{eqn:m=m-v-mse-1}, we obtain
\begin{equation}\label{eqn:T1-and-T2}
\E\bigl[\|v^t-F'(x^t)\|^2\bigr] \leq \cT_1 + \cT_2,
\end{equation}
where 
\begin{align*}
    \cT_1& =6m\sum_{i=1}^m \biggl(\sprod_{r\neq i}\ell_r^2 \biggr) \E\bigl[\|z_i^t-f_i'(y_{i-1}^t)\|^2\bigr], \\
    \cT_2 & =6m\sum_{i=1}^{m-1} \biggl(\sprod_{r\neq i+1}\!\!\!\ell_r^2 \biggr)
L_{i+1}^2\E\bigl[\|y_i^t-F_i^t\|^2\bigr].
\end{align*}
Next we bound the two terms~$\cT_1$ and $\cT_2$ separately.

For $\cT_1$, we use \eqref{eqn:m=m-z-mse} and set $B_i=B$ and $b_i=b$ for all 
$i=1,\ldots,m$, which yields
\begin{eqnarray*} 
\cT_1 
& \leq &  6m\sum_{i=1}^m\sprod_{r\neq i}\ell_r^2\cdot\left(\frac{\sigma_i^2}{B_i} +{L_i^2}\left(\sprod_{r=1}^{i-1}\ell_r^2\right)\frac{\tau\epsilon^2}{b_iL^2_F}\right)\\
&=&  6m\left(\frac{1}{B}\sum_{i=1}^m\biggl(\sprod_{r\neq i}\ell_r^2\biggr)\sigma_i^2 + \frac{\tau\epsilon^2}{b}\frac{1}{L^2_F} \sum_{i=1}^m L_i^2\left(\sprod_{r=1}^{i-1}\ell_r^4\right)\left(\sprod_{r=i+1}^m\ell_r^2\right)\right).
\end{eqnarray*}
Now recall the definition of $\sigma_F^2$ in~\eqref{eqn:sigma-F-delta-F}
and definition of $L_F$ in~\eqref{eqn:Lip-F-m=m}, which implies
\begin{equation}\label{eqn:L2-sum-squares}
    L_F^2 ~=~
    \left(\sum_{i=1}^m L_i\left(\sprod_{r=1}^{i-1}\ell_r^2\right)\left(\sprod_{r=i+1}^m\ell_r\right)\right)^{\! 2} 
    ~\geq~
    \sum_{i=1}^m L_i^2\left(\sprod_{r=1}^{i-1}\ell_r^4\right)\left(\sprod_{r=i+1}^m\ell_r^2\right) .
\end{equation}
Therefore, we have
\[
\cT_1 ~\leq~ 6m\left(\frac{\sigma_F^2}{B} + \frac{\tau\epsilon^2}{b}\right)
    ~=~ \frac{1}{m+1}\epsilon^2,
\]
where the equality is due to the choice of parameters 
in~\eqref{eqn:m=m-B-b}.

For $\cT_2$, recalling our notation $F_i^t=F_i(x^t)$, 
we use~\eqref{eqn:m=m-y-full} to obtain
\begin{align*}
\cT_2 
& \leq 6m\!\sum_{i=1}^{m-1}\biggl(\sprod_{r\neq i+1}\!\!\!\ell_r^2\biggr)L_{i+1}^2\cdot i\cdot\left[\sum_{r=1}^{i}\frac{\delta_r^2}{S_r}\biggl(\sprod_{j=r+1}^i\!\!\ell_j^2\biggr) + \biggl(\sprod_{j=1}^i\ell_j^2\biggr)\sum_{r=1}^i\frac{\tau\epsilon^2}{s_rL_F^2}\right]\nonumber\\
& = 6m\!\sum_{r=1}^{m-1}\!\frac{\delta_r^2}{S_r}\sum_{i = r}^{m-1}\! i\!\cdot\!\! \biggl(\sprod_{j\neq i+1}\!\!\!\ell_j^2 \!\cdot\! L_{i+1}^2 \!\cdot\!\!\!\sprod_{j=r+1}^i\!\!\!\ell_j^2\!\biggr) + 6m\!\sum_{r=1}^{m-1}\!\frac{\tau\epsilon^2}{s_rL_F^2}\!\sum_{i=r}^{m-1}\! i \!\cdot\! \biggl(\sprod_{j=1}^i\!\!\ell_j^4 \!\cdot\! L_{i+1}^2 \!\cdot\!\!\! \sprod_{j=i+1}^m\!\!\!\ell_j^2\!\biggr)\nonumber\\
& \leq  6m^2\!\sum_{r=1}^{m-1}\!\frac{\delta_r^2}{S_r}\sum_{i = r}^{m-1} \!\biggl(\sprod_{j\neq i+1}\!\!\!\ell_j^2 \!\cdot\! L_{i+1}^2 \!\cdot\!\!\!\sprod_{j=r+1}^i\!\!\! \ell_j^2\biggr)\! + 6m^2\!\sum_{r=1}^{m-1}\!\frac{\tau\epsilon^2}{s_rL_F^2}\!\sum_{i=r}^{m-1}\!\biggl(\sprod_{j=1}^i\!\ell_j^4\!\cdot\! L_{i+1}^2 \!\cdot\!\!\!\sprod_{j=i+1}^m\!\!\!\ell_j^2\!\biggr)\nonumber\\
& =  6m^2\!\sum_{r=1}^{m-1}\!\!\frac{\delta_r^2}{S_r\!\sprod_{j=1}^r\!\ell_r^2}\!\!\sum_{i=r}^{m-1} \!\biggl(\sprod_{j=i+1}^m\!\!\!\ell_j^2\!\cdot\! L_{i+1}^2 \!\cdot\!\!\sprod_{j=1}^i\!\!\ell_j^4\!\biggr)\! + 6m^2\!\sum_{r=1}^{m-1}\!\!\frac{\tau\epsilon^2}{s_rL_F^2}\!\!\sum_{i=r}^{m-1}\!\biggl(\sprod_{j=1}^i\!\!\ell_j^4 \!\cdot\! L_{i+1}^2 \!\!\cdot\!\!\!\!\!\sprod_{j=i+1}^m\!\!\!\!\ell_j^2\!\biggr)\nonumber\\
& \leq  6m^2\sum_{r=1}^{m-1}\left(\frac{\delta_r^2L_F^2}{S_r\sprod_{j=1}^r\ell_r^2} +\frac{\tau\epsilon^2}{s_r}\right), \nonumber
\end{align*}
where in the second inequality we used $i\leq m$ and in the last inequality
we used~\eqref{eqn:L2-sum-squares}.
Now setting $S_r=S$ and $s_r=s$ for $r=1,\ldots,m-1$, and using the definition
of $\delta_F^2$ in~\eqref{eqn:sigma-F-delta-F}, we have
\[
\cT_2 ~\leq~ 6m^2\frac{\delta_F^2}{S} + 6m^3\frac{\tau\epsilon^2}{s}
~=~ \frac{m}{m+1}\epsilon^2,
\]
where the last equality is due to the choice of parameters 
in~\eqref{eqn:m=m-S-s}.

Finally, with the above bounds on $\cT_1$ and $\cT_2$, 
we get from~\eqref{eqn:T1-and-T2} that for any $t\geq 0$,
\[
\E[\|v^t-F'(x^t)\|^2]
~\leq~ \cT_1 + \cT_2
~\leq~ \frac{1}{m+1}\epsilon^2 + \frac{m}{m+1}\epsilon^2
~=~ \epsilon^2.
\]
This finishes the proof.
\end{proof}

As a result, we have the following theorem.

\begin{theorem}\label{thm:m=m-stochastic}
Consider problem~\eqref{eqn:multi-level-stochastic} with $m\geq 2$, and
suppose Assumptions~\ref{assumption:Lip-F} and~\ref{assumption:m=m} hold. 
In Algorithm~\ref{alg:nested-spider-m}, if we set $\epsilon_k = \frac{m\theta L_F}{k\cdot\ell_F}$ with $\theta>0$ and set the other parameters as 
in Lemma~\ref{lemma:m=m-v-mse},  
then the output~$\bar{x}$ satisfies $\E[\|\cG(\bar{x})\|]\leq\epsilon$ after
$K = \tilde O\left(\frac{mL_F}{\epsilon\cdot\ell_F}\right)$ epochs, and the total sample complexity is
$\tilde O\bigl(m^4 L_F(\sigma_F^2+\delta_F^2+\ell_F^2)\epsilon^{-3}\bigr)$.
\end{theorem}
\begin{proof}
From Theorem~\ref{theorem:cvg-SNPAG} and Lemma~\ref{lemma:m=m-v-mse}, 
the output of Algorithm~\ref{alg:nested-spider-m} satisfies
\begin{equation} \label{thm:m-adp-1} 
	\E\bigl[\|\mathcal{G}(\bar{x})\|\bigr] \leq  \frac{4L_F(\Phi(x^0)-\Phi_*)}{\sum_{t=0}^{T-1}\epsilon_{\chi(t)}} + \frac{4\sum_{t=0}^{T-1}\epsilon^2_{\chi(t)}}{\sum_{t=0}^{T-1}\epsilon_{\chi(t)}}.  
\end{equation}
Note that we have 
	\begin{eqnarray*}
	\sum_{t=0}^{T-1}\epsilon_{\chi(t)} & = & \sum_{k=1}^K\tau_k\epsilon_k = \sum_{k=1}^K\frac{\ell_F}{2m\epsilon_k}\cdot\epsilon_k = \frac{\ell_FK}{2m},\\
	\sum_{t=0}^{T-1}\epsilon_{\chi(t)}^2 & = & \sum_{k=1}^K\tau_k\epsilon_k^2 = \sum_{k=1}^K\frac{\ell_F}{2m}\cdot\frac{m\theta L_F}{k\cdot\ell_F} \leq \theta L_F\ln (K).
	\end{eqnarray*} 
Substituting the above inequalities into \eqref{thm:m-adp-1} yields 
    $$\E\left[\|\cG(\bar x)\|\right] \leq \frac{8mL_F\big(\Phi(x^0) - \Phi_* + \theta\ln K\big)}{K\cdot \ell_F}\leq \tilde O\left(\frac{mL_F}{K\cdot\ell_F}\right).$$ 
Consequently, we can choose $K = \tilde O\left(\frac{mL_F}{\epsilon\cdot\ell_F}\right)$ such that $\E\left[\|\cG(\bar x)\|\right]\leq\epsilon$. With the $B_i^k, b_i^k, S_i^k, s_i^k$ defined in Lemma~\ref{lemma:m=m-v-mse}, the total sample complexity is 
    \begin{eqnarray}
    \sum_{k=1}^K\!\sum_{i=1}^m\bigl(B_i^k \!+\! (\tau_k\!-\!1)b_i^k\bigr)
    \!+\!\sum_{k=1}^K\!\sum_{i=1}^{m-1} \bigl(S_i^k \!+\! (\tau_k\!-\!1)s_i^k\bigr) \!& = &\! O\!\biggl(\!\frac{mK^3\bigl(\sigma_F^2+\delta_F^2+\ell_F^2/m\bigr)}{L_F^2/\ell_F^2}\!\biggr)\nonumber\\
    \!& = &\! \tilde O\biggl(\!\frac{m^4 L_F(\sigma_F^2+\delta_F^2+\ell_F^2)}{\epsilon^3}\!\biggr)\nonumber.
    \end{eqnarray}
This finishes the proof.
\end{proof}

Next, we present our results of the finite-sum optimization problem.
\begin{theorem}\label{thm:m=m-finite-sum}
	Consider problem~\eqref{eqn:multi-level-finite-sum} with $m\geq 2$, and suppose Assumptions~\ref{assumption:Lip-F} and~\ref{assumption:m=m}.(a) hold. In addition, let $N_\mathrm{max}=\max\{N_1,\ldots,N_m\}$ and assume the target precision $\epsilon$ satisfies $\sqrt{N_\mathrm{max}} \leq \frac{\ell_F}{2m\epsilon}.$
In Algorithm~\ref{alg:nested-spider-m}, we set $\eta=1/2L_F$. For $k = 1,...,K$ and $i = 1,...,m$, we set $\epsilon_k = \big(\frac{\theta L_F}{k\cdot\sqrt{N_{max}}}\big)^{1/2}$ , $\theta>0$, $\tau_k\equiv\tau=\sqrt{N_\mathrm{max}}$ and
\[
	b_i^k\equiv\min\Big\{6m(m+1)\sqrt{N_\mathrm{max}},\, N_i\Big\}, \qquad
	s_i^k\equiv \min\Big\{6m^2(m+1)\sqrt{N_\mathrm{max}},\, N_i\Big\}. 
\]
Moreover, for every~$t$ at the beginning of each epoch, let $\cB_i^t=\cS_i^t=\{1,\ldots,N_i\}$ be the full batches for $i=1,\ldots,m.$
Then the output $\bar{x}$ satisfies $\E[\|\cG(\bar{x})\|]\leq\epsilon$ 
after $K= \tilde O\bigl(\frac{L_F}{\sqrt{N_{max}}\cdot\epsilon^{2}}\bigr)$ epochs, and the sample complexity is 
$O\bigl(\sum_{i=1}^m N_i + m^4 L_F\sqrt{N_\mathrm{max}}\epsilon^{-2}\bigr)$.
\end{theorem}

\begin{proof}
	Through a similar line of proof of Lemma \ref{lemma:m=m-v-mse}, $\E\big[\|v^t - F'(x^t)\|^2\big]\leq\epsilon_{\chi(t)}^2$ still holds. Consequently \eqref{thm:m-adp-1} also holds in this case. Note that 
	\begin{eqnarray*}
		\sum_{t=0}^{T-1}\epsilon_{\chi(t)} & = & \sum_{k=1}^K\tau_k\epsilon_k = \sum_{k=1}^K\sqrt{N_{max}}\cdot\sqrt{\frac{\theta L_F}{k\cdot\sqrt{N_{max}}}} \geq 2N_{max}^{\frac{1}{4}}\sqrt{(K+1)\cdot\theta L_F},\\
		\sum_{t=0}^{T-1}\epsilon_{\chi(t)}^2 & = & \sum_{k=1}^K\tau_k\epsilon_k^2 = \sum_{k=1}^K\sqrt{N_{max}}\cdot\frac{\theta L_F}{k\cdot\sqrt{N_{max}}}  \leq 2\theta L_F \ln(K).
	\end{eqnarray*} 
	Substitute the above inequalities into \eqref{thm:m-adp-1}, we get $\E\left[\|\cG(\bar x)\|\right] \leq \tilde O\big((\frac{L_F}{\sqrt{N_{max}}\cdot K})^\frac{1}{2}\big)$. As result, it suffices to set the number of epochs $K = \tilde O\big(\frac{L_F}{\sqrt{N_{max}}\cdot\epsilon^{2}}\big)$ to guarantee $\E\left[\|\cG(\bar x)\|\right] \leq\epsilon$. Consequently, the total sample complexity is 
	\begin{eqnarray}
	\sum_{k=1}^K\!\sum_{i=1}^m\bigl(N_i \!+\! (\tau_k\!-\!1)b_i^k\bigr)
	\!+\!\sum_{k=1}^K\!\sum_{i=1}^{m-1} \bigl(N_i \!+\! (\tau_k\!-\!1)s_i^k\bigr) \!& = &\!\tilde O\!\left(\!\sum_{i=1}^m N_i \!+\! m^4 L_F\sqrt{N_\mathrm{max}}\epsilon^{-2}\!\right)\nonumber.
	\end{eqnarray}
\end{proof}

If the condition that $\sqrt{N_\mathrm{max}} \leq \frac{\ell_F}{2m\epsilon}$ does not hold, then
we can use the setting in Theorem~\ref{thm:m=m-stochastic} and obtain
the same order or slightly better sample complexity.

\begin{remark}
	\label{remark:m=m-stochastic}
	In Theorem \ref{thm:m=m-stochastic}, if we set $\epsilon_k\equiv\epsilon$, then the Algorithm output $\bar x$ s.t. $\E[\|\cG(\bar x)\|]\leq\epsilon$ after $K = O(\frac{mL_F}{\epsilon\cdot\ell_F})$ epochs, and the total sample complexity is $O(m^4 L_F(\sigma_F^2+\delta_F^2+\ell_F^2)\epsilon^{-3})$. Similarly, in Theorem \ref{thm:m=m-finite-sum}, if we set $\epsilon_k\equiv\epsilon$, then the Algorithm output $\bar x$ s.t. $\E[\|\cG(\bar x)\|]\leq\epsilon$ after $K = O(\frac{L_F}{\sqrt{N_{max}}\cdot\epsilon^2})$ epochs, and the total sample complexity is $O\left(\sum_{i=1}^m N_i + m^4 L_F\sqrt{N_\mathrm{max}}\epsilon^{-2}\right)$. Although using constant $\epsilon_k\equiv\epsilon$ improves the sample complexities of Theorem \ref{thm:m=m-stochastic} and \ref{thm:m=m-finite-sum} by a logarithmic factor, the $O(\epsilon)$ step length may be too consevative in the early stages. In practice, the adaptive setting of $\epsilon_k$ in Theorem \ref{thm:m=m-stochastic} and \ref{thm:m=m-finite-sum} can be much more efficient. 
\end{remark}

\section{Numerical experiments}
\label{sec:numerical}

In this section, we present numerical experiments to demonstrate the effectiveness of the proposed algorithms and compare them with related work.
We first apply NPAG with different variance-reduced estimators on a noncovex sparse classification problem. Since this is a one-level finite-sum problem, we also compare it with ProxSARAH \cite{ProxSARAH2019}.
Then we consider a sparse portfolio selection problem, which has a two-level composition structure. We compare the Nested-\textsc{SPIDER} method with several other methods for solving two-level problems. 

\subsection{Sparse binary classification}

We consider the following $\ell_1$-regularized empirical risk minimization problem:
\begin{equation}
\label{prob:Sparse-Binary}
\min_{w} \frac{1}{N}\sum_{i=1}^N \ell(a_i^Tx,b_i) + \beta\|x\|_1,
\end{equation}
where $\{(a_i,b_i)\}_{i=1}^N$ are $N$ labeled examples with feature vectors $a_i\in\R^d$ and binary labels $y_i\in\{-1,,+1\}$. 
Specifically, we consider two different loss functions: 
the logistic difference loss 
$\ell_{ld}(t,b) = \log\big(1 + e^{-bt}\big) - \log\big(1 + e^{-bt-1}\big)$, 
and the 2 layer neural network loss 
$\ell_{nn}(t,b) = \big(1 - \frac{1}{1 + e^{-bt}}\big)^2$.  
For the $\ell_1$-norm penalty, we set $\beta = 1/N$.

This is a standard one-level finite-sum problem, corresponding to the case of $m=1$ in~\eqref{eqn:multi-level-finite-sum}.
We compare the ProxSPIDER method (Algorithm~\ref{alg:prox-spider}) with the ProxSARAH algorithm \cite{ProxSARAH2019}. 
ProxSARAH can be implemented with two schemes: for the variable stepsize scheme (Theorem 5, \cite{ProxSARAH2019}) we use the name ProxSARAH-vstp; for the constant stepsize scheme (Theorem 6, \cite{ProxSARAH2019}) we call it ProxSARAH-cstp.
We also include the SVRG and SAGA estimator with normalized stepsize rule as described in Section \ref{subsec:SVRG} and \ref{subsec:SAGA} respectivley, which we call ProxSVRG and ProxSAGA.

\begin{figure}[t]
	\centering 
	\includegraphics[width=0.45\linewidth]{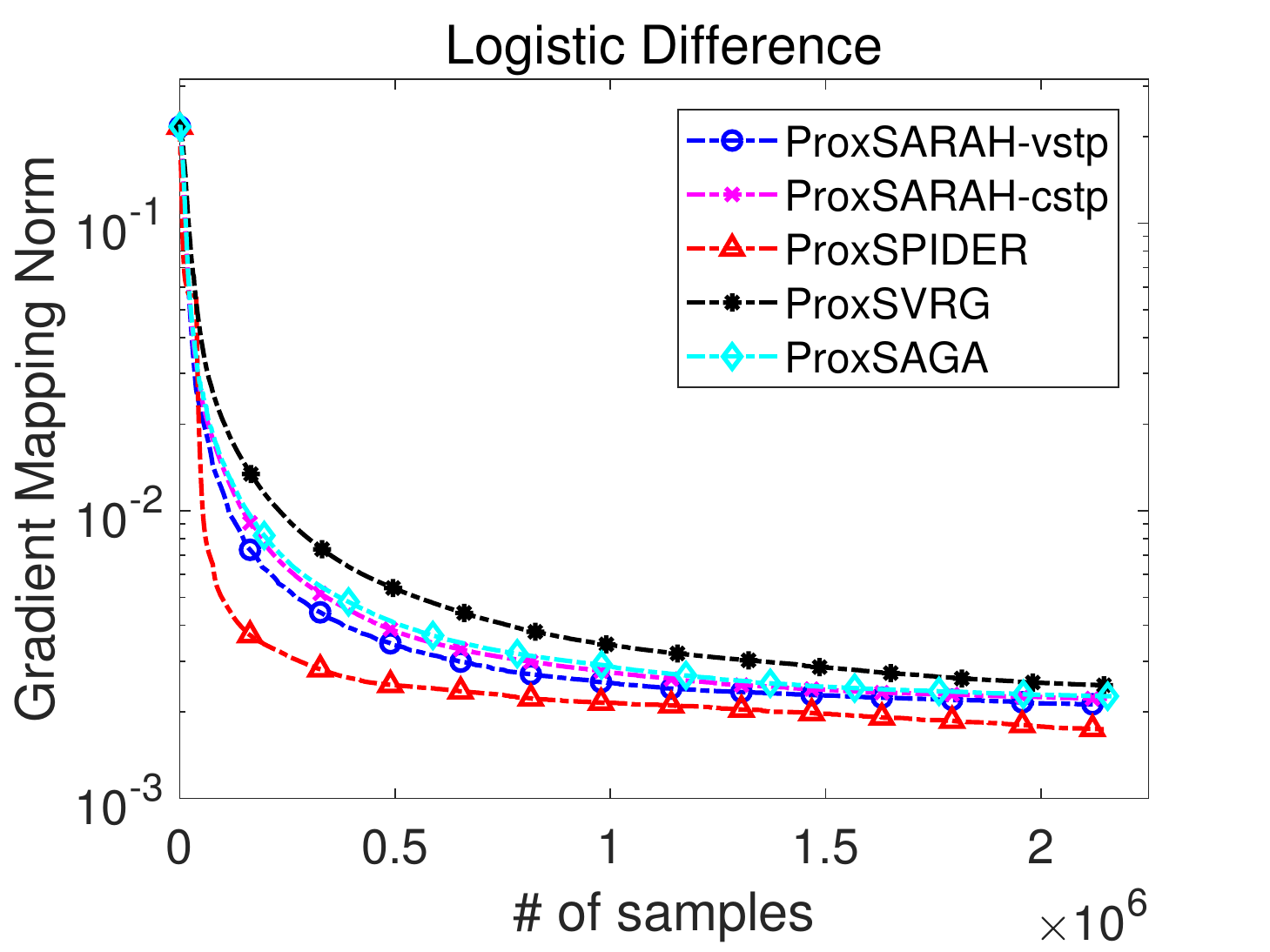}
	\hspace{2ex} 
	\includegraphics[width=0.45\linewidth]{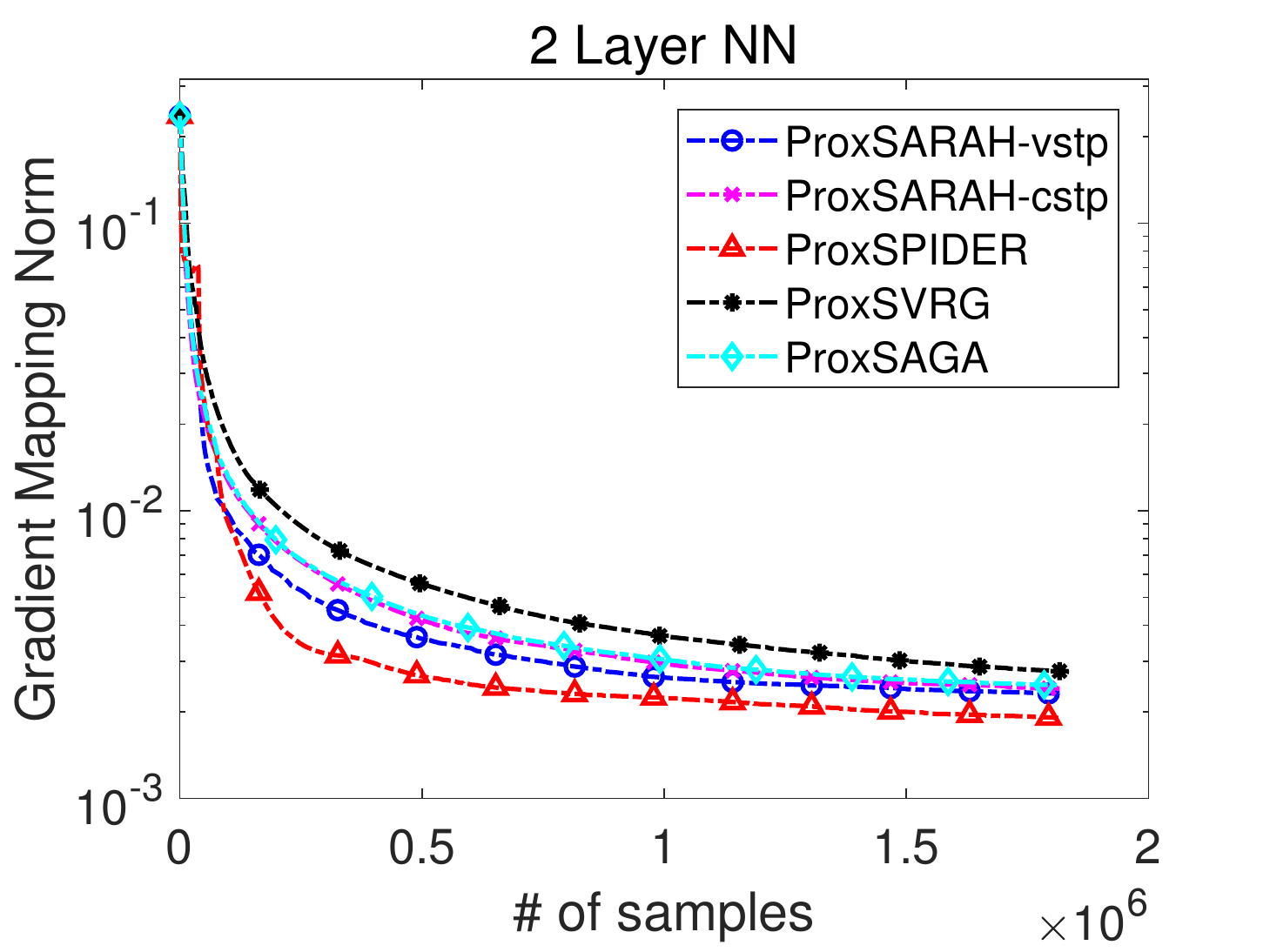} 
    \\[2ex]
	\includegraphics[width=0.45\linewidth]{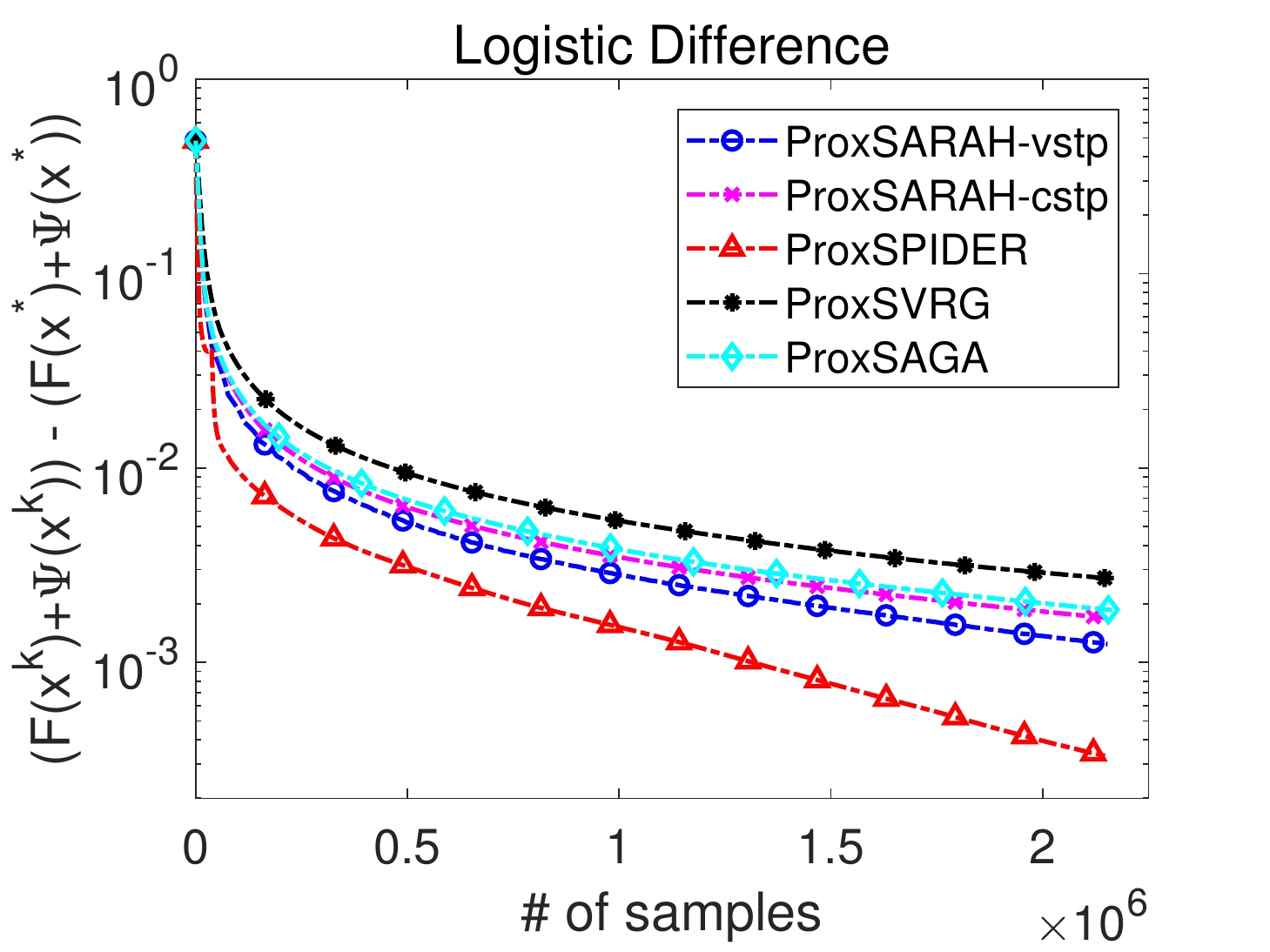}
	\hspace{2ex} 
	\includegraphics[width=0.45\linewidth]{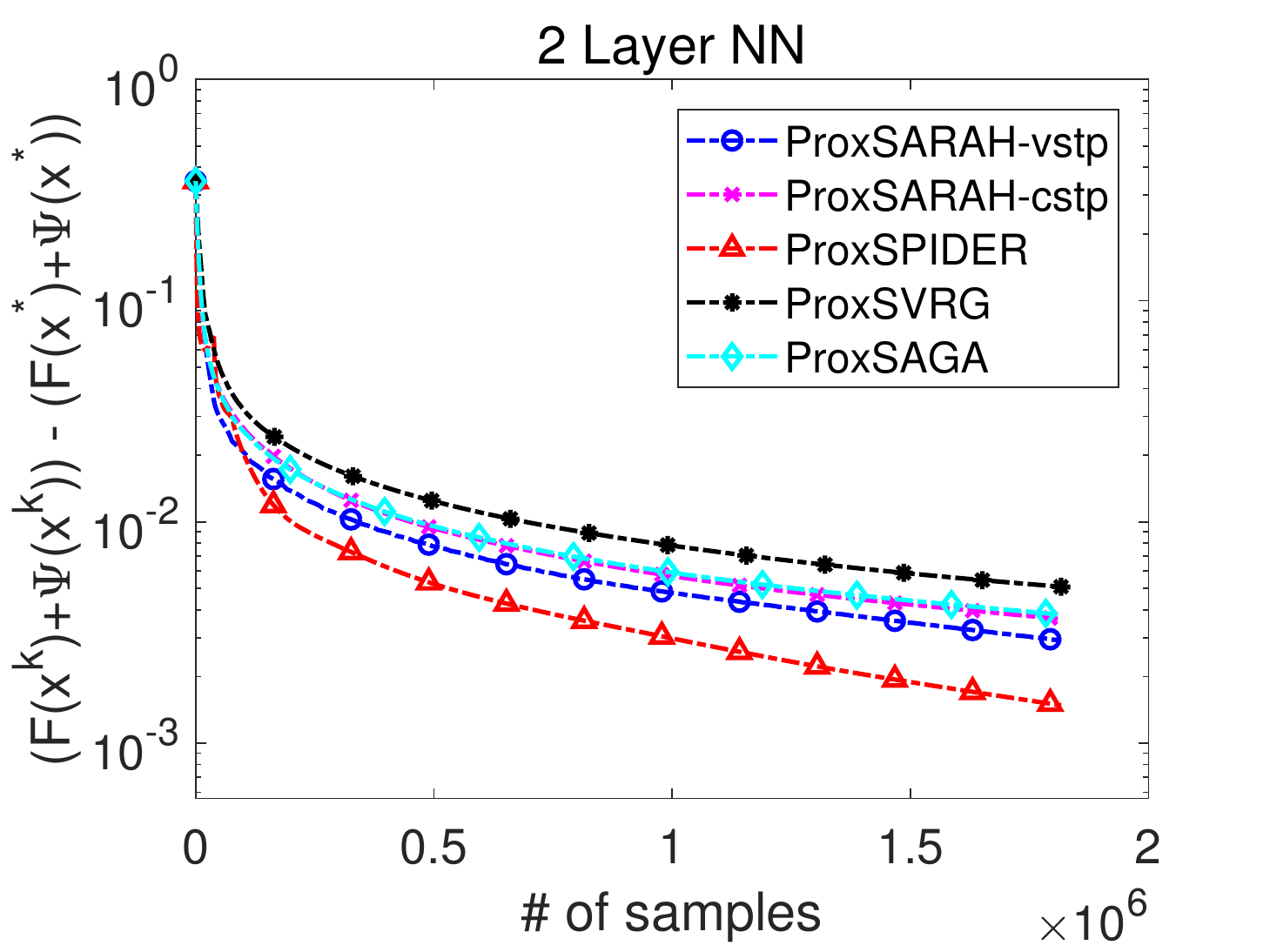} 
	\caption{Experiments on sparse binary classification on \emph{mnist} dataset, $\beta = 1/12691$.}
	\label{fig:mnist}
\end{figure}

The parameter setting (batch size, epoch length, etc.) of ProxSARAH-vstp and ProxSARAH-cstp are made to be consistent with the ProxSARAH-A-v1 and ProxSARAH-v1 in \cite{ProxSARAH2019} where each mini-batch consists of only 1 sample, this is the parameter setting that yields the best empirical performance in \cite{ProxSARAH2019}. 
For ProxSPIDER, we set both the batchsize and epoch length to be $\big\lceil\sqrt{N}\big\rceil$. For ProxSVRG and ProxSAGA, we set the batchsize to be $\big\lceil N^\frac{2}{3}\big\rceil$, the epoch length to be $\big\lceil N^\frac{1}{3}\big\rceil$.  Finally, for ProxSPIDER, ProxSVRG and ProxSAGA, we use a variable sequence $\epsilon_k$ as in NPAG (Algorithm~\ref{alg:NPAG} and set $\epsilon_k = 10/\sqrt{k}$. 

First, we compare these methods over the \emph{mnist}\footnote{http://yann.lecun.com/exdb/mnist/} data sets. To fit the binary classification problems, we extract the data of two arbitrary digits, which are 1 and 9 in our experiments. This subset of mnist consists $N=12691$ data points. 
For the ProxSARAH-vstp and ProxSARAH-cstp, the parameter to be tuned is the estimator of the Lipschtiz constant $L$. The smaller $L$ is, the larger the stepsize is. For the ProxSPIDER, ProxSVRG and ProxSAGA, the parameter to be tuned is the stepsize $\eta$, which corresponds to $1/L$ in this experiment. 
We choose $L$ among $\{0.01,0.1,1,10\}$ and it turns out that $L = 1$ works best for ProxSARAH-vstp and ProxSARAH-cstp for both loss functions. 
We choose $\eta$ among $\{0.1, 1, 10, 100\}$ and it turns out that $\eta = 1$ works best for ProxSIDER, ProxSVRG and ProxSAGA. 
These choices are consistent with the relationship $\eta=1/L$.

Figure~\ref{fig:mnist} shows the norm of the gradient mapping and objective value gap of different methods.
For all the compared methods, the curves are averaged over 20 runs of the algorithm.
It can be seen that these different methods have similar performance, with ProxSVRG and ProxSAGA slightly worse and ProxSPIDER slightly better than other methods. In theory ProxSARAH and ProxSPIDER have the same sampling complexity, and the empirical results are consistent with the theory.
Since the plots are very similar for the norm of gradient mapping and the objective value gap, we will only present plots of the gradient mapping norms for the rest of the experiments.

We also conduct our experiment on the \emph{rcv1.binary}\footnote{https://www.csie.ntu.edu.tw/~cjlin/libsvmtools/datasets/binary.html} dataset, which is a medium sized dataset with $N = 20242$ data points and each data point has $d=47236$ features. 
The parameters are chosen in a similar way to the experiments on \emph{mnist} dataset. 
In particular, $L = 0.01$ works best for ProxSARAH-vstp and ProxSARAH-cstp for both loss functions and $\eta = 100$ works best for ProxSIDER, ProxSVRG and ProxSAGA (again we have $\eta=1/L$).
Figure \ref{fig:rcv1-binary'} shows the results of comparison. 
In this particular case, ProxSAGA perform the best despite a worse sample complexity in theory.

\begin{figure}[t]
	\centering 
	\includegraphics[width=0.45\linewidth]{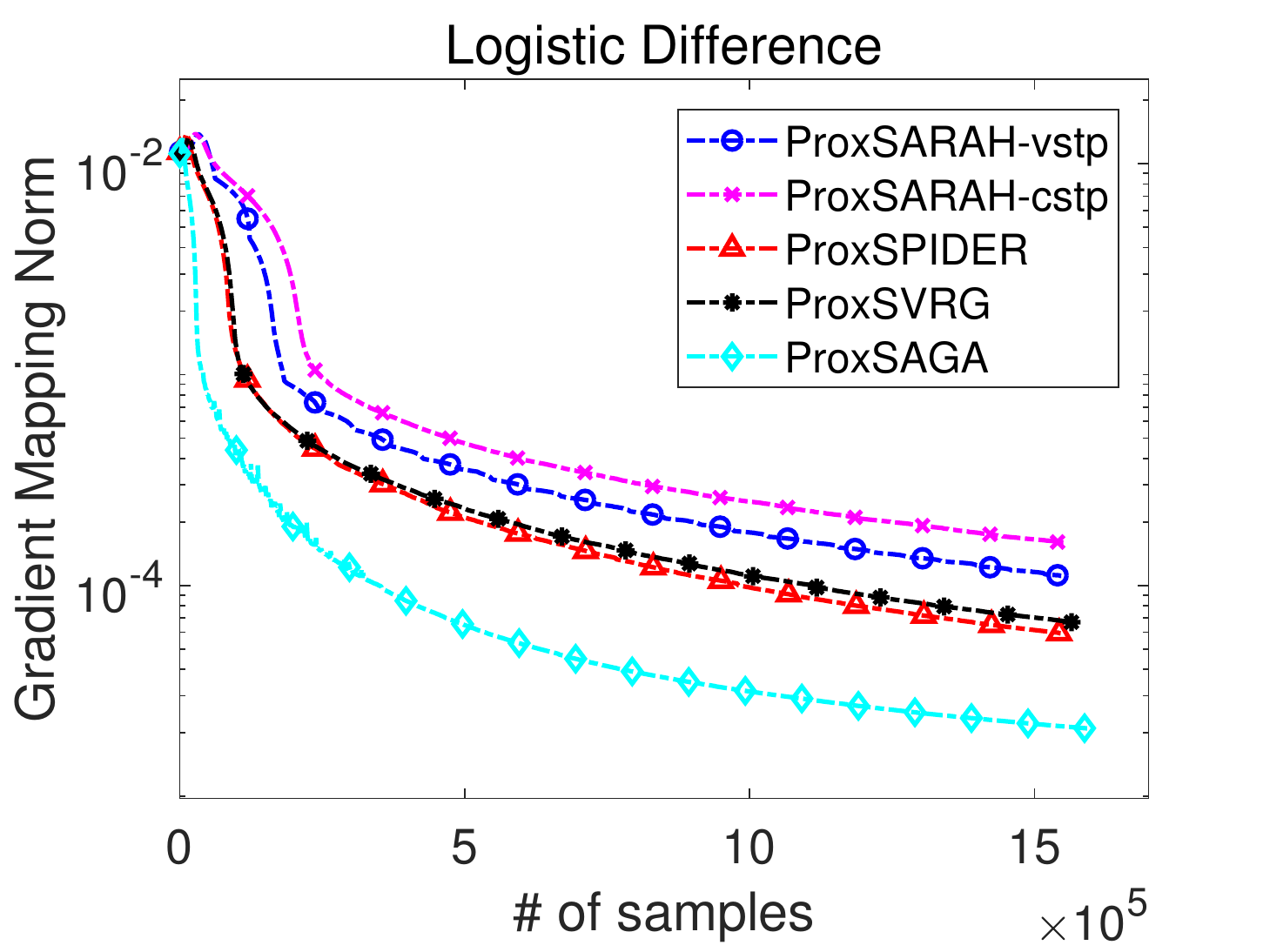}
    \hspace{2ex}
	\includegraphics[width=0.45\linewidth]{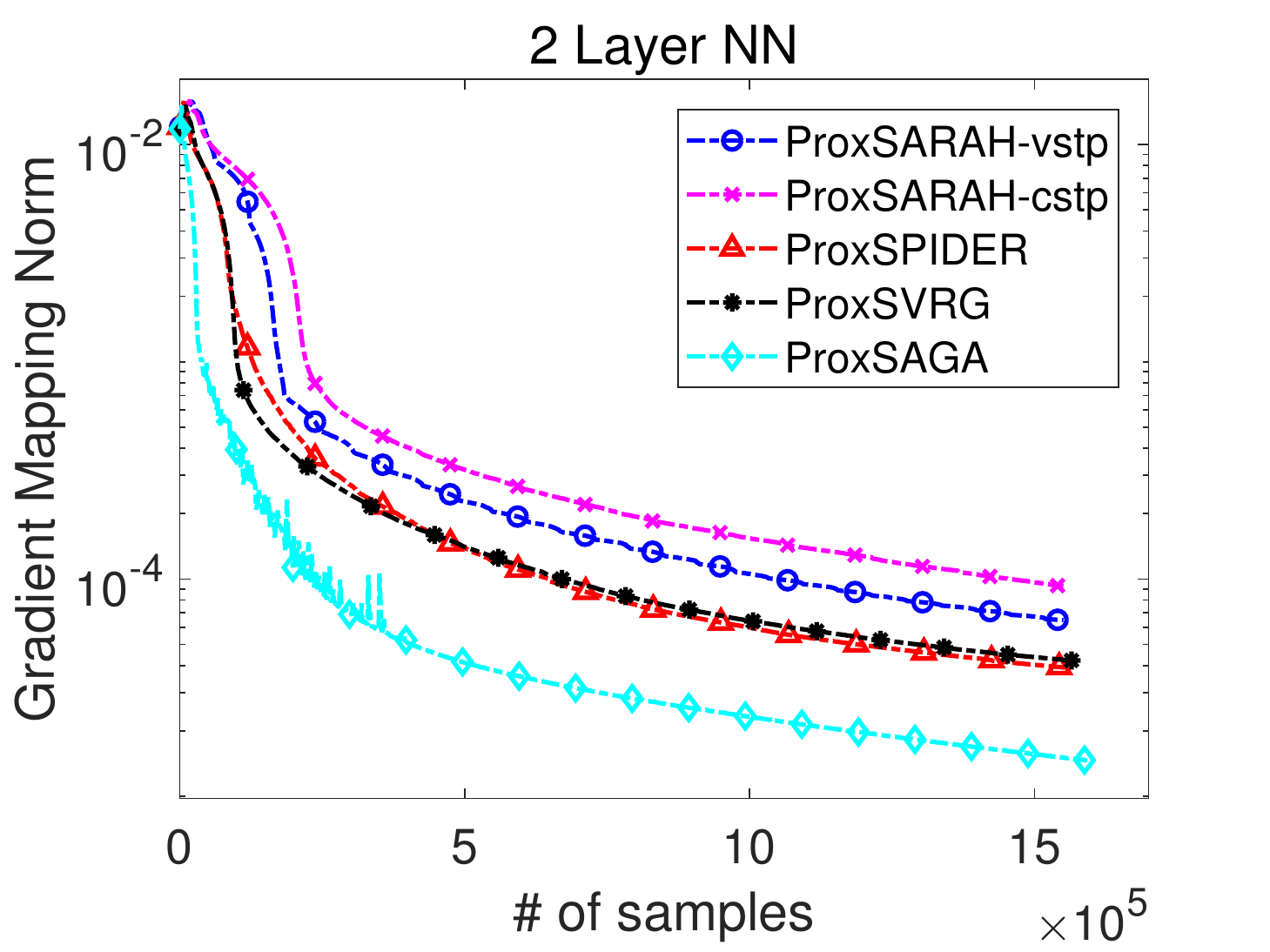} 
	\caption{Experiments on sparse binary classification on \emph{rcv1} dataset, $\beta = 1/20242$.}
	\label{fig:rcv1-binary'}
\end{figure}

\subsection{Sparse portfolio selection problem}
In this section, we present the numerical results for a risk-averse portfolio optimization problems, which is a common benchmark example used in many stochastic composite optimization methds ($e.g.$, \cite{VRSC-PG,ZhangXiao2019C-SAGA,ZhangXiao2019CIVR,TY.Lin}). Consider the scenario where $d$ assets can be invested during the time periods $\{1, ..., N\}.$ We denote the per unit payoff vector at time $i$ by $r_i:=[r_{i,1},...,r_{i,d}]^T$, with $r_{i,j}$ being the per unit payoff of asset $j$ at time period $i$. Let $x\in\R^d$ be the decision variable, where the $j$-th component $x_j$ represents the amount of investment
to asset $j$ throughout the $N$ time period, for $j = 1, . . ., d.$ 

In the portfolio optimization problems, our aim is to maximize the
average payoff over $N$ periods. However, a penalty is imposed on the variance of the payoff across the $N$ periods, which accounts for the risk of the investment. In addition, we also apply an $\ell_1$ penalty to impose sparsity on the portfolio $x$, which gives the following formulation:
$$\maximize_{x\in\R^d}\E\big[\langle r_\xi,x\rangle\big] - \lambda\Var(\langle r_\xi,x\rangle) - \beta\|x\|_1,$$ 
where the random variable $\xi$ follows the uniform distribution over $\{1,...,N\}$.
Using the identity that $\Var(X) = \E[X^2] - \E[X]^2$, the problem can be reformulated as 
\begin{equation}
\label{prob:portfolio}
\minimize_{x\in\R^d}  \E\big[\lambda\cdot\langle r_\xi,x\rangle^2 - \langle r_\xi,x\rangle\big] - \lambda\cdot\E\big[\langle r_\xi,x\rangle\big]^2 + \beta\|x\|_1.
\end{equation} 
If we set 
$f_{1,\xi}(x) : \R^d\mapsto\R^2 = [\langle r_\xi,x\rangle\,\,\langle r_\xi,x\rangle^2]^T$, $f_{2,\xi_2}(y,z) : \R^2\mapsto\R \equiv  -y - \lambda y^2 + \lambda z$, and $\Psi(x) = \beta\|x\|_1,$ then \eqref{prob:portfolio} becomes a special case of \eqref{eqn:multi-level-stochastic} with $m = 2$. 

In the experiments, we compare our method with the C-SAGA algorithm \cite{ZhangXiao2019C-SAGA}, the CIVR and CIVR-adp algorithms \cite{ZhangXiao2019CIVR}, and the VRSC-PG algorithm \cite{VRSC-PG}. We test these algorithms with the real world portfolio datasets from Keneth R. French Data Library\footnote{http://mba.tuck.dartmouth.edu/pages/faculty/ken.french/data\_library.html}. The first dataset is named  \texttt{Industrial}-38, which contains 38 industrial assets payoff over 2552 consecutive time periods; and the second dataset is named  \texttt{Industrial}-49, and contains 49 industrial assets over 24400 consecutive time periods. 
For both datasets, we set $\lambda = 0.2$ and $\beta = 0.01$  in problem \eqref{prob:portfolio}. All the curves are averaged over 20 runs of the methods and are plotted against the number of samples of the component function and gradient evaluations. 

\begin{figure}[t]
	\centering 
	\includegraphics[width=0.45\linewidth]{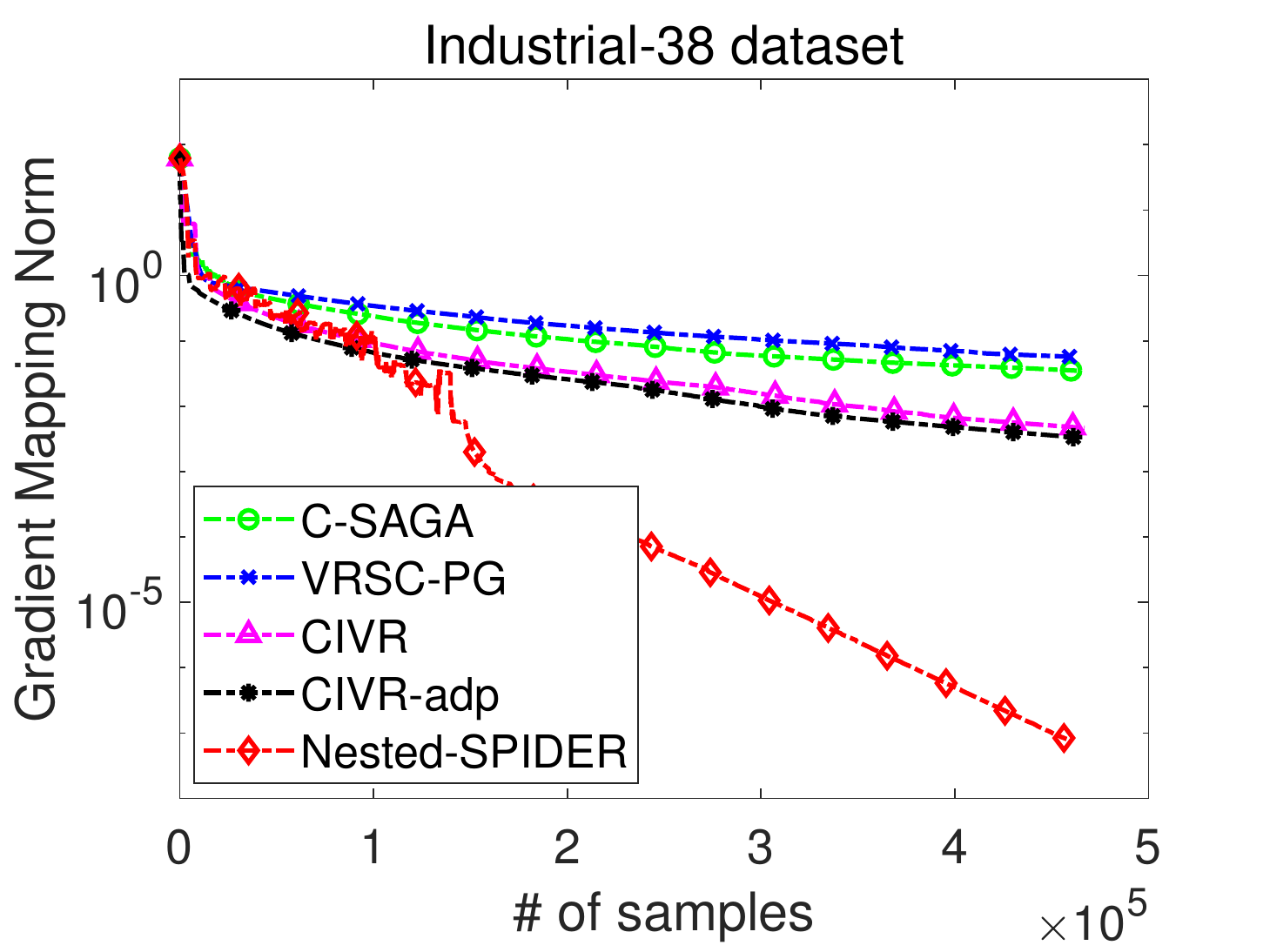}
    \hspace{2ex}
	\includegraphics[width=0.45\linewidth]{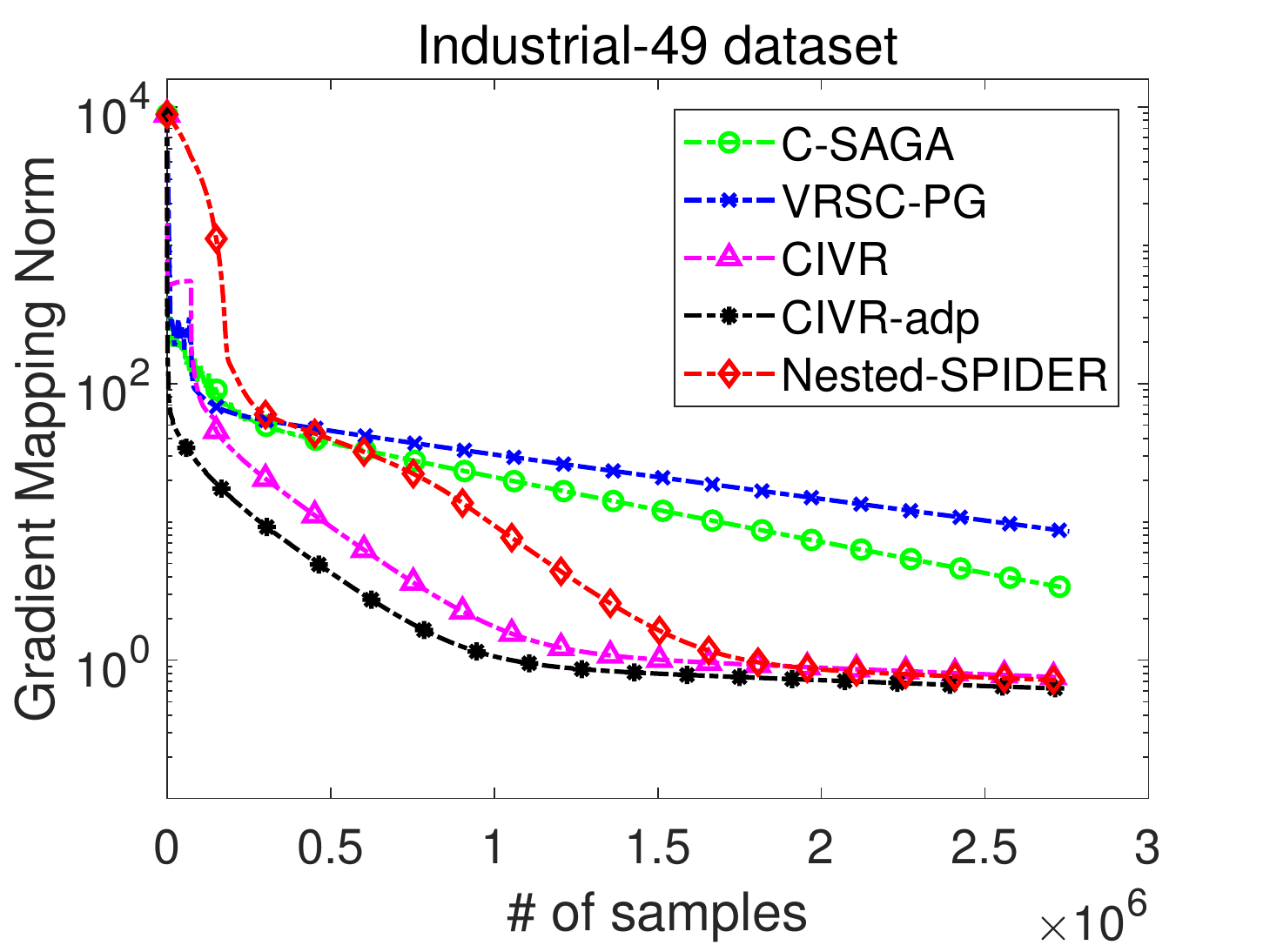}
	\caption{Numerical experiments on the sparse portfolio optimization problem.}
	\label{fig:risk-averse}
\vspace{1ex}
\end{figure}

\begin{figure}[t]
	\centering 
	\includegraphics[width=0.45\linewidth]{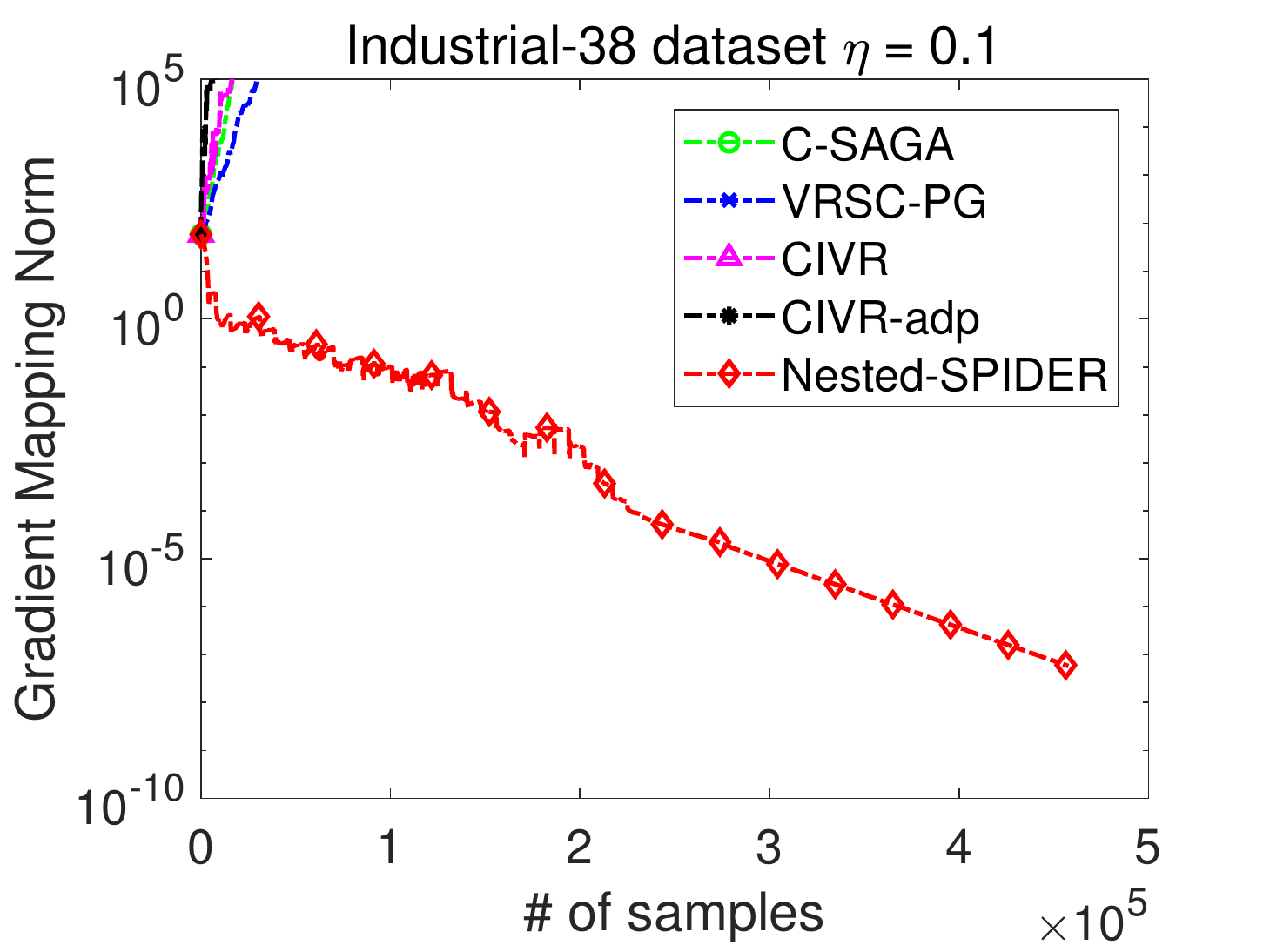}
    \hspace{2ex}
	\includegraphics[width=0.45\linewidth]{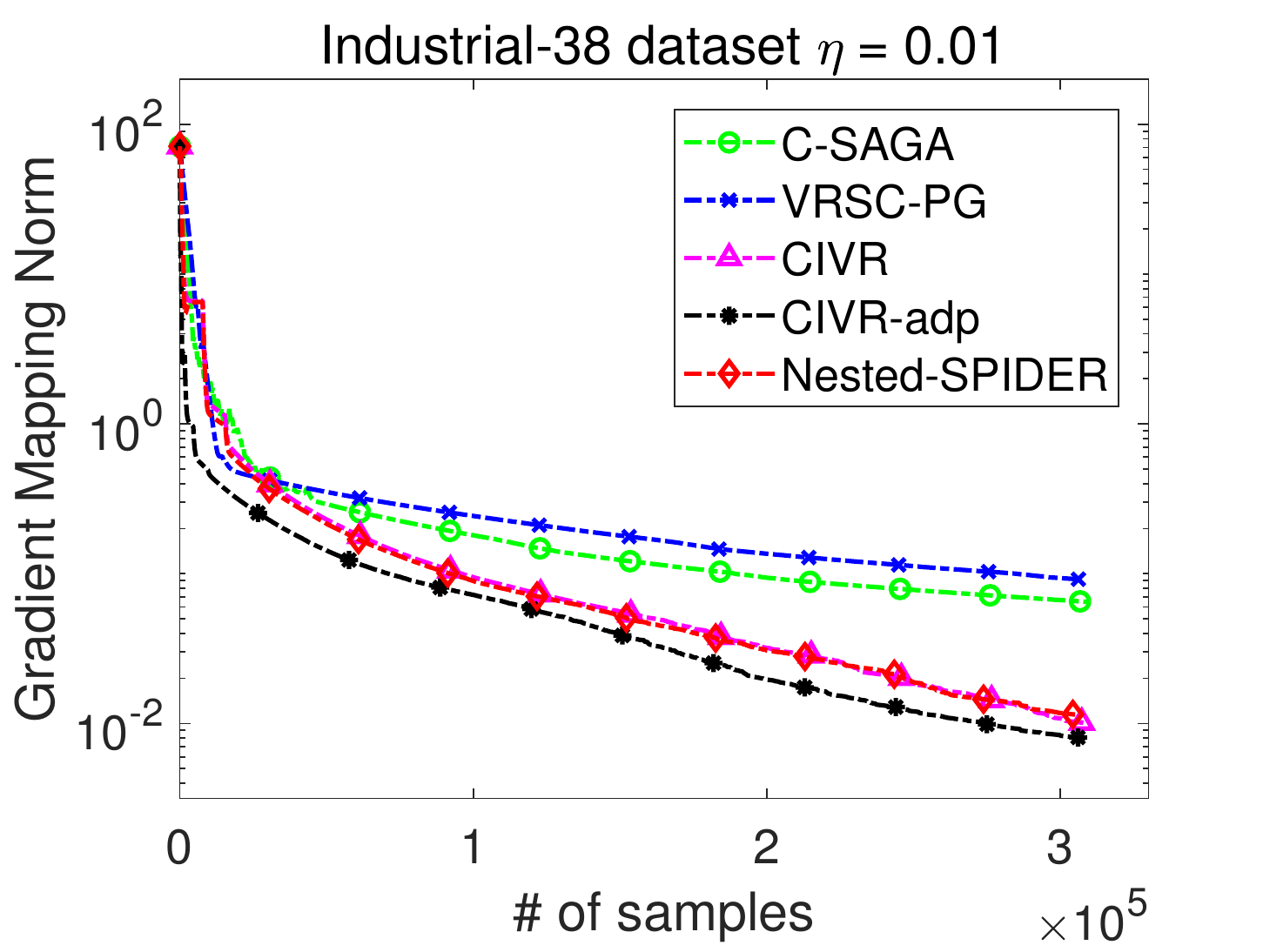}
	\caption{Experiments on \texttt{Industrial}-38 datasets using different step size parameter $\eta$.}
	\label{fig:risk-averse'}
\end{figure}

For these experiments, CIVR and Nested-SPIDER both take the batch size of $\lceil N^{1/2}\rceil$; CIVR-adp takes the adaptive batch size of $S_k = \lceil\min\{10k+1,N^{1/2}\}\rceil$; VRSC-PG and C-SAGA both take the batch size of $\lceil N^{2/3}\rceil$. For all these methods, their stepsizes $\eta$ are chosen as the best from $\{0.0001,0.001,0.01,0.1,1\}$ by experiments. In addition, we set $\epsilon_k = C/\sqrt{k}$ for Nested-SPIDER with $C$ chosen from $\{0.1,1,10,50\}$ by experiments. For the \texttt{Industrial}-38 dataset, $\eta = 0.01$ works best for CIVR, CIVR-adp, C-SAGA and VRSC-PG, and $\eta = 0.1$ and $C = 1$ work best for Nested-SPIDER. For the \texttt{Industrial}-49 dataset, $\eta = 0.0001$ works best for all five tested algorithms and $C = 50$ works best for the Nested-SPIDER algorithm.  

Figure \ref{fig:risk-averse} shows that for the \texttt{Industrial}-49 dataset, the CIVR-adp and CIVR perform the best. However, Nested-SPIDER quickly catches up with them after a slightly slower early stage, which is due to the relatively restrictive stepsizes caused by normalization (in early stages of the algorithm). On the contrary, Nested-SPIDER converges much faster than all other methods in the \texttt{Industrial}-38 dataset because the normalization step enables us to exploit a significantly larger stepsizes (in later stages of the algorithm), under which all the other methods diverge.  In Figure \ref{fig:risk-averse'}, we show the following behavior of the Nested-SPIDER on \texttt{Industrial}-38. On the left, we set $\eta = 0.1$ for all the algorithms and $C=1$ for Nested-SPIDER. All methods except for Nested-SPIDER diverge. On the right, we set $\eta = 0.01$ for all the algorithms and $C=50$ for Nested-SPIDER. In this case, Nested-SPIDER behaves almost identical to CIVR.

\section{Conclusion}
\label{sec:conclusion}

We have proposed a normalized proximal approximate gradient (NPAG) method
for solving multi-level composite stochastic optimization problems.
The approximate gradients at each iteration are obtained via nested
variance reduction using the SARAH/\textsc{Spider} estimator.
In order to find an approximate stationary point where the expected norm of
its gradient mapping is less than $\epsilon$, the total sample complexity of 
our method is $O(\epsilon^{-3})$ in the expectation case, 
and $O(N+\sqrt{N}\epsilon^{-2})$ in the finite-sum case where $N$ is the 
total number of functions across all composition levels.
In addition, the dependence of the sample complexity on the number of 
composition levels is polynomial, rather than exponential as in previous work.
Our results rely on a uniform Lipschitz condition on the composite mappings
and functions, which is stronger than 
the mean-squared Lipschitz condition required for obtaining similar 
complexities for single-level stochastic optimization.

The NPAG method extends the \textsc{Spider}-SFO
method \cite{SPIDER2018NeurIPS} to the proximal setting and for solving
multi-level composite stochastic optimization problems.
In particular, requiring a uniform bound on the MSEs of the approximate 
gradients allows separate analysis of the optimization algorithm and
variance reduction. 
In addition, the uniform bound on the step lengths at each iteration 
is the key to enable our analysis of nested variance reduction.
More flexible rules for choosing the step sizes or step lengths have been 
developed to obtain similar sample complexities for single-level
stochastic or finite-sum optimization problems, especially those used in the
\textsc{Spider}-Boost \cite{SpiderBoost2018} 
and Prox-SARAH \cite{ProxSARAH2019} methods.
It is likely that such step size rules may also be developed for multi-level 
problems, although their complexity analysis may become much more involved.

\appendix
\section{Proof of Lemma~\ref{lem:m=m-Lip-constants}}
\label{sec:Lip-F-m-proof} 
\begin{proof}
Under Assumption~\ref{assumption:m=m}.(a), 
it is straightforward to show that $f_i = \E_{\xi_i}f_{i,\xi_i}$ 
is $\ell_i$-Lipschitz and its gradient $f'_i = \E_{\xi_i}f'_{i,\xi_i}$ 
is $L_i$-Lipschitz. Recall the definitions
$F_i = f_i\circ f_{i-1}\circ\cdots\circ f_1$ for $i=1,\ldots,m$. By the chain rule, we have
\begin{eqnarray*}
	& & \|F'(x) - F'(y)\|\\
	& = & \bigl\|[f_1'(x)]^T\cdots[f_{m-1}'(F_{m-2}(x))]^T f_m'(F_{m-1}(x)) \nonumber\\
    && ~-~ [f_1'(y)]^T\cdots[f_{m-1}'(F_{m-2}(y))]^T f_m'(F_{m-1}(y))\bigr\|\nonumber\\
	& \leq  &\big\|[f_1'(x)]^T\!\!\cdots[f_{m\!-\!1}'(F_{m\!-\!2}(x))]^T \! f_m'(F_{m\!-\!1}(x)) - [f_1'(y)]^T[f_{2}'(F_{1}(x))]^T\!\!\cdots f_m'(F_{m\!-\!1}(x))\big\|\nonumber\\
	&& + \cdots  \\
    && + \big\|[f_1'(y)]^T\!\!\cdots\![f_{m\!-\!1}'(F_{m\!-\!2}(y))]^T \!f_m'(F_{m\!-\!1}(x)) - [f_1'(y)]^T[f_{2}'(F_{1}(y))]^T\!\!\cdots\! f_m'(F_{m\!-\!1}(y))\big\|\nonumber\\
	& \leq & L_1\sprod_{r\neq1}\ell_r\|x-y\| + L_2\sprod_{r\neq 2}\ell_r\|F_1(x)-F_1(y)\| + ... + L_m\sprod_{r\neq m}\ell_r\|F_m(x)-F_m(y)\|.
\end{eqnarray*}

On the other hand, we have for any $x$ and $y$ that  
\begin{eqnarray*}
	\|F_i(x)-F_i(y)\| &	= &\|f_i(F_{i-1}(x))-f_i(F_{i-1}(y))\| \\
	& \leq & \ell_i\|F_{i-1}(x)-F_{i-1}(y)\| \leq \sprod_{r=1}^i\ell_r\|x-y\|.
\end{eqnarray*}
Therefore $F_i$ 
is $\ell_F$-Lipschitz with $\ell_F=\sprod_{i=1}^m\ell_i$. Substituting these Lipschitz constants into the bound of  $\|F'(x) - F'(y)\|$ yields
\[
\|F'(x) - F'(y)\|\leq \left(\sum_{i=1}^mL_i\biggl(\sprod_{r=1}^{i-1}\ell_r\biggr)^2\cdot\biggl(\sprod_{r=i+1}^{m}\ell_r\biggr)\right)\|x-y\|, 
\]
which gives the expression of $L_F$ in~\eqref{eqn:Lip-F-m=m}.
\end{proof}

\bibliographystyle{plain}
\bibliography{NestedSpider}

\end{document}